\documentclass[english,11pt]{article}
\usepackage[T1]{fontenc}
\usepackage[latin9]{inputenc}
\usepackage{amsmath}
\usepackage{amssymb}
\usepackage{amsfonts}
\usepackage{babel}
\usepackage{fancyhdr}
\usepackage{color}
\usepackage[toc]{appendix}

\makeatletter
\newtheorem{theorem}{Theorem}

\newtheorem{condition}[theorem]{Condition}

\newtheorem{definition}[theorem]{Definition}

\newtheorem{lemma}[theorem]{Lemma}
\newtheorem{construction}[theorem]{Construction}

\numberwithin{equation}{section}
\numberwithin{example}{section}
\numberwithin{theorem}{section}

\newenvironment{proof}[1][Proof]{\noindent\textbf{#1.} }{\ \rule{0.5em}{0.5em}}
\makeatother
\input{tcilatex}
\begin{document}

\title{Moderate deviations for recursive stochastic algorithms}
\author{Paul Dupuis\thanks{%
Research supported in part by the Department of Energy (DE-SCOO02413), the
National Science Foundation (DMS-1317199), and the Army Research Office
(W911NF-12-1-0222).} \ and Dane Johnson\thanks{%
Research supported in part by the Department of Energy (DE-SCOO02413).} \\
Division of Applied Mathematics \\
Brown University \\
Providence, RI 02912 }
\maketitle

\begin{abstract}
We prove a moderate deviation principle for the continuous time
interpolation of discrete time recursive stochastic processes. \ The methods
of proof are somewhat different from the corresponding large deviation
result, and in particular the proof of the upper bound is more complicated.
The results can be applied to the design of accelerated Monte Carlo
algorithms for certain problems, where schemes based on moderate deviations
are easier to construct and in certain situations provide performance
comparable to those based on large deviations.
\end{abstract}

\section{Introduction}

In this paper we consider $\mathbb{R}^{d}$-valued discrete time processes of
the form 
\begin{equation*}
X_{i+1}^{n}=X_{i}^{n}+\frac{1}{n}b(X_{i}^{n})+\frac{1}{n}\upsilon
_{i}(X_{i}^{n})\text{, }X_{0}^{n}=x_{0},
\end{equation*}%
where $\{\upsilon _{i}(\cdot )\}_{i\in \mathbb{N}_{0}}$ are zero mean random
independent and identically distributed (iid) vector fields, and focus on
their continuous time piecewise linear interpolations $\{X^{n}(t)\}_{0\leq
t\leq T}$ with $X^{n}(i/n)=X_{i}^{n}$ (see (\ref{eq:pw_linear}) for the
precise definition). \ Under certain conditions there is a law of large
number limit $X^{0}\in C([0,T]:\mathbb{R}^{d})$, and the large deviations of 
$X^{n}$ from this limit have been studied extensively (see, e.g., \cite%
{azerug,demzei,dupell4,frewen,gar}). Here we introduce a scaling $a(n)$
satisfying $a(n)\rightarrow 0$ and $a(n)\sqrt{n}\rightarrow \infty $, and
study the amplified difference between $X^{n}$ and its noiseless version $%
X^{n,0}$ (see Section 2 for the definition of $X^{n,0}$): 
\begin{equation*}
Y^{n}=a(n)\sqrt{n}(X^{n}-X^{n,0}).
\end{equation*}%
Under Condition \ref{cond:main} introduced below $\sup_{t\in \lbrack
0,T]}\left\Vert X^{0}(t)-X^{n,0}(t)\right\Vert \sim O(1/n)$, and hence this
will behave the same asymptotically as $a(n)\sqrt{n}(X^{n}-X^{0})$ \ We
demonstrate, under weaker conditions on the noise $\upsilon _{i}(\cdot )$
than are necessary when considering $X^{n}$, that $Y^{n}$ satisfies the
large deviation principle on $C([0,T]:\mathbb{R}^{d})$ with a
\textquotedblleft Gaussian\textquotedblright\ type rate function. \ As is
customary for this type of scaling, we refer to this as moderate deviations.

To demonstrate this result we prove the equivalent Laplace principle, which
involves evaluating limits of quantities of the form 
\begin{equation*}
a(n)^{2}\log E\left[ \exp \left\{ -\frac{1}{a(n)^{2}}F(Y^{n})\right\} \right]
\end{equation*}%
when $F$ is bounded and continuous. This is done by representing each of
these quantities in terms of a stochastic control problem, and then using
weak convergence methods as in \cite{dupell4}. Key results needed in this
approach are establishing tightness of controls and controlled processes,
and identifying their limits.

While one might expect the proof of this moderate deviations result to be
similar to the corresponding large deviations result, there are important
differences. For example, the tightness proof is significantly more
complicated in the case of moderate deviations than it is in the case of
large deviations. \ For large deviations one is able to establish an a
priori bound on certain relative entropy costs associated with any sequence
of nearly minimizing controls, and under this boundedness of the relative
entropy costs, the empirical measures of the controlled driving noises as
well as the controlled processes are tight. \ However, owing to the scaling
in moderate deviations, even with the information that the analogous
relative entropy costs decay like $O(1/a(n)^{2}n)$, tightness of the
empirical measures of the noises does not hold. \ Instead, one must consider
empirical measures of the conditional means of the noises, and additional
effort is required for the law of large numbers type result that shows that
the conditional means are adequate to determine the limit. \ This extra
difficulty arises for moderate deviations (even with the vanishing relative
entropy costs), because the noise itself is being amplified by $a(n)\sqrt{n}$%
.

A second way in which the proofs for large and moderate deviations differ is
in their treatment of degenerate noise, i.e., problems where the support of $%
\upsilon _{i}(\cdot )$ is not all of $\mathbb{R}^{d}$. This leads to
significant difficulties in the proof of the large deviation lower bound,
and requires a delicate and involved mollification argument. In contrast,
the proof in the setting of moderate deviations, though more involved than
the nondegenerate case, is much more straightforward.

As a potential application of these results we mention their usefulness in
the design and analysis of Monte Carlo schemes for events whose probability
is small but not very small. For such problems the performance of standard
Monte Carlo may not be adequate, especially if the quantity must be computed
for many different parameter settings, as in say an optimization problem.
Then accelerated Monte Carlo may be of interest, and as is well known such
schemes (e.g., importance sampling and splitting) benefit by using
information contained in the large deviation rate function as part of the
algorithm design (e.g., \cite{blaglyled, deadup, dupwan3, dupwan5}). \ In a
situation where one considers events of small but not too small probability
one may find the moderate deviation approximation both adequate and
relatively easy to apply, since moderate deviations lead to situations where
the objects needed to design an efficient scheme can be explicitly
constructed in terms of solutions to the linear-quadratic regulator. These
issues will be explored elsewhere.

The existing literature on moderate deviations considers various settings. \
Baldi \cite{bal2} considers the same scaling used here but with no state
dependence. \ For the empirical measure of a Markov chain, de Acosta \cite%
{dea3} and de Acosta and Chen \cite{chedea} prove lower and upper bounds,
respectively. \ Guillin \cite{gui} considers inhomogeneous functionals of a
\textquotedblleft fast\textquotedblright\ continuous time ergodic Markov
chain, and in \cite{gui2} this is extended to a small noise diffusion whose
coefficients depend on the \textquotedblleft fast\textquotedblright\ Markov
chain. \ There are also results for martingale differences such as Dembo 
\cite{dem2}, Gao \cite{gao}, and Djellout \cite{dje}. For various reasons,
the issues previously mentioned regarding the difficulties in the proof of
the upper bound and the simplification in the lower bound for degenerate
noise do not play a role in these papers.

The paper is organized as follows. \ Section 2 gives the statement of the
problem and notation. \ Section 3 contains the proof of tightness and the
characterization of limits, which account for most of the mathematical
difficulties, and are also the main results needed to prove the Laplace
principle. \ Sections 4 and 5 give the proofs of the upper and lower Laplace
bounds. Although all proofs are given for the time interval $[0,1]$, they
extend with only notational differences to $[0,T]$ for any $T\in (0,\infty )$%
.

\section{Background and Notation}

Let 
\begin{equation*}
X_{i+1}^{n}=X_{i}^{n}+\frac{1}{n}b(X_{i}^{n})+\frac{1}{n}\upsilon
_{i}(X_{i}^{n})\text{, }X_{0}^{n}=x_{0}
\end{equation*}%
where the $\{\upsilon _{i}(\cdot )\}_{i\in \mathbb{N}_{0}}$ are zero mean
iid vector fields with distribution given by the stochastic kernel $\mu _{x}$%
. Thus if $\mathcal{B}(\mathbb{R}^{d})$ is the Borel $\sigma $-algebra on $%
\mathbb{R}^{d}$, then $x\rightarrow \mu _{x}(B)$ is measurable for all $B\in 
\mathcal{B}(\mathbb{R}^{d})$, $\mu _{x}(\cdot )$ is a probability measure on 
$\mathcal{B}(\mathbb{R}^{d})$ for all $x\in \mathbb{R}^{d}$, and $P(\upsilon
_{i}(x)\in B)=\mu _{x}(B)$ for all $x\in \mathbb{R}^{d}$, $B\in \mathcal{B}(%
\mathbb{R}^{d})$ and $i\in \mathbb{N}_{0}$. Define 
\begin{equation*}
H_{c}(x,\alpha )\doteq \log \left( \dint_{\mathbb{R}^{d}}e^{\left\langle
y,\alpha \right\rangle }\mu _{x}(dy)\right)
\end{equation*}%
for $\alpha \in \mathbb{R}^{d}$. The subscript $c$ reflects the fact that
this log moment generating function uses the centered distribution $\mu _{x}$%
, rather than the usual $H(x,\alpha )=H_{c}(x,\alpha )+\left\langle \alpha
,b(x)\right\rangle $. We will use the following.

\begin{condition}
\label{cond:main}

\begin{itemize}
\item There exists $\lambda >0$ and $K_{\text{mgf}}<\infty $ such that 
\begin{equation}
\sup_{x\in \mathbb{R}^{d}}\sup_{\left\Vert \alpha \right\Vert \leq \lambda
}H_{c}(x,\alpha )\leq K_{\text{mgf}}\text{.}  \label{MGF Bound}
\end{equation}

\item $x\rightarrow \mu _{x}(dy)$ is continuous with respect to the topology
of weak convergence.

\item $b(x)$ is continuously differentiable, and the norm of both $b(x)$ and
its derivative are uniformly bounded by some constant $K_{b}<\infty $.
\end{itemize}
\end{condition}

Throughout this paper we let $\left\Vert \alpha \right\Vert
_{A}^{2}=\left\langle \alpha ,A\alpha \right\rangle $ for any $\alpha \in 
\mathbb{R}^{d}$ and symmetric, nonnegative definite matrix $A$. \ Define%
\begin{equation*}
A_{ij}(x)\doteq \dint_{\mathbb{R}^{d}}y_{i}y_{j}\mu _{x}(dy),
\end{equation*}%
and note that the weak continuity of $\mu _{x}$ with respect to $x$ and (\ref%
{MGF Bound}) ensure that $A\left( x\right) $ is continuous in $x$ and its
norm is uniformly bounded by some constant $K_{A}$. \ Note that 
\begin{equation*}
\frac{\partial H_{c}(x,0)}{\partial \alpha _{i}}=\dint_{\mathbb{R}%
^{d}}y_{i}\mu _{x}(dy)=0
\end{equation*}%
and%
\begin{equation*}
\frac{\partial ^{2}H_{c}(x,0)}{\partial \alpha _{i}\partial \alpha _{j}}%
=\dint_{\mathbb{R}^{d}}y_{i}y_{j}\mu _{x}(dy)=A_{ij}(x)
\end{equation*}%
for all $i,j\in \{1,\ldots ,d\}$ and $x\in \mathbb{R}^{d}$, and that $A(x)$
is nonnegative-definite and symmetric. \ For $x\in \mathbb{R}^{d}$ we can
therefore write 
\begin{equation*}
A(x)=Q(x)\Lambda (x)Q^{T}(x),
\end{equation*}%
where $Q(x)$ is an orthogonal matrix whose columns are the eigenvectors of $%
A(x)$ and $\Lambda (x)$ is the diagonal matrix consisting of the eigenvalues
of $A(x)$ in descending order. \ In what follows we define $\Lambda ^{-1}(x)$
to be the diagonal matrix with diagonal entries equal to the inverse of the
corresponding eigenvalue for the positive eigenvalues, and equal to $\infty $
for the zero eigenvalues. \ Then when we write%
\begin{equation}
\left\Vert \alpha \right\Vert _{A^{-1}(x)}^{2}=\left\Vert \alpha \right\Vert
_{Q(x)\Lambda ^{-1}(x)Q^{T}(x)}^{2},  \label{eq:DiffInvDecomp}
\end{equation}%
we mean a value of $\infty $ for $\alpha \in \mathbb{R}^{d}$ not in the
linear span of the eigenvectors corresponding to the positive eigenvalues,
and the standard value for vectors $\alpha \in \mathbb{R}^{d}$ in that
linear span. \ Assumption (\ref{MGF Bound}) implies there exists some $%
K_{DA}<\infty $ and $\lambda _{DA}\in (0,\lambda ]$ (independent of $x$)
such that%
\begin{equation}
\sup_{x\in \mathbb{R}^{d}}\sup_{\left\Vert \alpha \right\Vert \leq \lambda
_{DA}}\max_{i,j,k}\left\vert \frac{\partial ^{3}H_{c}(x,\alpha )}{\partial
\alpha _{i}\partial \alpha _{j}\partial \alpha _{k}}\right\vert \leq \frac{%
K_{DA}}{d^{3}},  \label{eq:H3rdDerivBnd}
\end{equation}%
and consequently for all $\left\Vert \alpha \right\Vert \leq \lambda _{DA}$
and all $x\in \mathbb{R}^{d}$ 
\begin{equation}
\frac{1}{2}\left\Vert \alpha \right\Vert _{A(x)}^{2}-\left\Vert \alpha
\right\Vert ^{3}K_{DA}\leq H_{c}(x,\alpha )\leq \frac{1}{2}\left\Vert \alpha
\right\Vert _{A(x)}^{2}+\left\Vert \alpha \right\Vert ^{3}K_{DA}\text{.}
\label{eq:H-bound}
\end{equation}

Define the continuous time linear interpolation of $X_{i}^{n}$ by $%
X^{n}(i/n)=X_{i}^{n}$ for $i=0,...,n$ and 
\begin{equation}
X^{n}(t)=(i+1-nt)X_{i}^{n}+(nt-i)X_{i+1}^{n}  \label{eq:pw_linear}
\end{equation}%
for $t\in (i/n,i/n+1/n)$. \ In addition, define 
\begin{equation*}
X_{i+1}^{n,0}=X_{i}^{n,0}+\frac{1}{n}b\left( X_{i}^{n,0}\right) \text{, }%
X_{0}^{n,0}=x_{0}
\end{equation*}%
and let $X^{n,0}(t)$ be the analogously defined continuous time linear
interpolation. \ Clearly $X^{n,0}(t)\rightarrow X^{0}(t)$ in $C([0,1]:%
\mathbb{R}^{d})$, where 
\begin{equation*}
X^{0}(t)=\int_{0}^{t}b(X^{0}(s))ds+x_{0}\text{.}
\end{equation*}%
Since $E\upsilon _{i}(x)=0$ for all $x\in \mathbb{R}^{d}$, we know that $%
X^{n}(t)\rightarrow X^{0}(t)$ in $C([0,1]:\mathbb{R}^{d})$ in probability. \
One can estimate probabilities for events involving paths outside the law of
large numbers limit $X^{0}$ by proving a large deviation principle and
finding the corresponding rate function. \ 

\begin{definition}
\bigskip Let $\{Z^{n}$, $n\in 
\mathbb{N}
\}$ be a sequence of random variables defined on a probability space $%
(\Omega ,\tciFourier ,P)$ and taking values in a Polish space $\mathcal{Z}$.
\ A function $I:\mathcal{Z}\rightarrow \lbrack 0,\infty ]$ is called a rate
function if for any $M<\infty $ the set $\{x:I(x)\leq M\}$ is compact in $%
\mathcal{Z}$. \ The sequence $\{Z^{n}\}$ satisfies the large deviation
principle on $\mathcal{Z}$\ with rate function $I$ and sequence $r(n)$ if
the following two conditions hold.

\begin{itemize}
\item Large Deviation Upper Bound: for each closed subset $F$ of $\mathcal{Z}
$\ 

\begin{equation*}
\limsup_{n\rightarrow \infty }r(n)\log P(Z^{n}\in F)\leq -\inf_{z\in F}I(z).
\end{equation*}

\item Large Deviation Lower Bound: for each open subset $G$ of $\mathcal{Z}$%
\ 

\begin{equation*}
\liminf_{n\rightarrow \infty }r(n)\log P(Z^{n}\in G)\geq -\inf_{z\in G}I(z).
\end{equation*}
\end{itemize}
\end{definition}

Under significantly stronger assumptions, including the assumption%
\begin{equation*}
\sup_{x\in \mathbb{R}^{d}}\sup_{\alpha \in \mathbb{R}^{d}}H_{c}(x,\alpha
)<\infty ,
\end{equation*}%
it has been shown that $X^{n}(t)$ satisfies the large deviation principle on 
$C([0,1]:\mathbb{R}^{d})$ with sequence $r(n)=1/n$ and rate function 
\begin{align*}
I_{L}(\phi )& =\inf \left\{ \int_{0}^{1}L_{c}(\phi (s),u(s))ds:\phi \left(
t\right) =x_{0}+\int_{0}^{t}b(\phi (s))ds\right. \\
& \quad \qquad \quad \left. +\int_{0}^{t}u(s)ds,t\in \lbrack 0,1]\right\} .
\end{align*}%
where 
\begin{equation*}
L_{c}(x,\beta )=\sup_{\alpha \in \mathbb{R}^{d}}\{\left\langle \alpha ,\beta
\right\rangle -H_{c}(x,\alpha )\}
\end{equation*}%
is the Legendre transform of $H_{c}(x,\alpha )$ \cite{dupell4, wen1, wen2,
wen3, wen4}.

Assume $a(n)$ satisfies 
\begin{equation}
a(n)\rightarrow 0\text{ and }a(n)\sqrt{n}\rightarrow \infty \text{.}
\label{a(n) behavior}
\end{equation}%
We define the rescaled difference 
\begin{equation*}
Y^{n}(t)=a(n)\sqrt{n}(X^{n}(t)-X^{n,0}(t)).
\end{equation*}%
As noted in the introduction, the result stated below also holds with the
interval $[0,1]$ replaced by $[0,T]$, $T\in (0,\infty )$. Let $D$ denote the
gradient operator.

\begin{theorem}
Assume Condition \ref{cond:main}. Then $\{Y^{n}\mathbb{\}}_{n\in \mathbb{N}}$
satisfies the large deviation principle on $C([0,1]:\mathbb{R}^{d})$ with
sequence $a(n)^{2}$ and rate function 
\begin{align*}
I_{M}(\phi )& =\inf \left\{ \frac{1}{2}\int_{0}^{1}\left\Vert
u(t)\right\Vert ^{2}dt:\phi (t)=\int_{0}^{t}Db(X^{0}(s))\phi (s)ds\right. \\
& \quad \qquad \quad \left. +\int_{0}^{t}A^{1/2}(X^{0}(s))u(s)ds,t\in
\lbrack 0,1]\right\} .
\end{align*}
\end{theorem}

$I_{M}$ is essentially the same as what one would obtain by using a linear
approximation around the law of large numbers limit $X^{0}$ of the dynamics
and a quadratic approximation of the costs in $I_{L}$. \ To prove the LDP,
it suffices to show the Laplace principle \cite[Theorem 1.2.3]{dupell4}%
\begin{align}
& \lim_{n\rightarrow \infty }-a(n)^{2}\log E\left[ e^{-\frac{1}{a(n)^{2}}%
F(Y^{n})}\right]  \notag \\
& \quad =\inf_{u\in L^{2}([0,1]:\mathbb{R}^{d})}\left\{ \frac{1}{2}%
\int_{0}^{1}\left\Vert u(s)\right\Vert ^{2}ds+F\left( \phi
^{A^{1/2}(X^{0})u}\right) \right\}  \label{eq:Laplace Principle}
\end{align}%
where%
\begin{equation}
\phi ^{u}\left( t\right) =\int_{0}^{t}Db(X^{0}(s))\phi
^{u}(s)ds+\int_{0}^{t}u(s)ds\text{.}  \label{phiDefined}
\end{equation}%
Note that 
\begin{equation*}
Y_{i+1}^{n}=Y_{i}^{n}+\frac{a(n)}{\sqrt{n}}\left(
b(X_{i}^{n})-b(X_{i}^{n,0})\right) +\frac{a(n)}{\sqrt{n}}\upsilon
_{i}(X_{i}^{n}),\quad Y_{0}^{n}=0
\end{equation*}

For $\eta ,\mu \in \mathcal{P}(\mathbb{R}^{d})$ [the set of probability
measures on $\mathcal{B}(\mathbb{R}^{d})$] , the relative entropy of $\eta $
with respect to $\mu $ is defined by 
\begin{equation*}
R(\left. \eta \right\Vert \mu )\doteq \int_{\mathbb{R}^{d}}\log \left( \frac{%
d\eta }{d\mu }(x)\right) \eta (dx)\in \lbrack 0,\infty ]
\end{equation*}%
if $\eta $ is absolutely continuous with respect to $\mu $, and $R(\left.
\eta \right\Vert \mu )\doteq \infty $ otherwise. For general properties of
relative entropy we refer to \cite[Section 1.4]{dupell4}. The variational
formula \cite[Proposition 1.4.2(a)]{dupell4} and chain rule \cite[Theorem
C.3.1]{dupell4}\ imply that%
\begin{equation}
-a(n)^{2}\log E\left[ e^{-\frac{1}{a(n)^{2}}F(Y^{n})}\right] =\inf_{\eta }E%
\left[ \dsum\limits_{i=0}^{n-1}a(n)^{2}R(\left. \eta _{i}\right\Vert \mu _{%
\bar{X}_{i}^{n}})+F(\bar{Y}^{n})\right]  \label{eq:rep}
\end{equation}%
for any bounded, continuous $F:C([0,1]:\mathbb{R}^{d})\rightarrow \mathbb{R}$%
. Here $\eta \in \mathcal{P}((\mathbb{R}^{d})^{n})$ is the joint
distribution of $(\bar{\upsilon}_{0},\ldots ,\bar{\upsilon}_{n-1})$, $\eta
_{i}(\cdot )$ is the conditional distribution on $\bar{\upsilon}_{i}$ given $%
(\bar{\upsilon}_{0},\ldots ,\bar{\upsilon}_{i-1})$,%
\begin{equation}
\bar{X}_{i+1}^{n}=\bar{X}_{i}^{n}+\frac{1}{n}b(\bar{X}_{i}^{n})+\frac{1}{n}%
\bar{\upsilon}_{i},\quad \bar{X}_{0}^{n}=x_{0},  \label{eq:defXbar}
\end{equation}%
\begin{equation}
\bar{Y}_{i+1}^{n}=\bar{Y}_{i}^{n}+\frac{a(n)}{\sqrt{n}}\left( b(\bar{X}%
_{i}^{n})-b(X_{i}^{n,0})\right) +\frac{a(n)}{\sqrt{n}}\bar{\upsilon}%
_{i},\quad \bar{Y}_{0}^{n}=0  \label{eq:defYbar}
\end{equation}%
and, similar to (\ref{eq:pw_linear}), $\bar{X}^{n}(t)$ and $\bar{Y}^{n}(t)$
are the continuous time linear interpolations of $\{\bar{X}%
_{i}^{n}\}_{i=0,\ldots ,n}$ and $\{\bar{Y}_{i}^{n}\}_{i=0,\ldots ,n}$. \
Note that $\eta _{i}$ depends on past values of the noise, but we suppress
this dependence in the notation. \ We will prove (\ref{eq:Laplace Principle}%
) by proving the lower bound%
\begin{align}
& \liminf_{n\rightarrow \infty }-a(n)^{2}\log E\left[ e^{-\frac{1}{a(n)^{2}}%
F(Y^{n})}\right]  \notag \\
& \quad \geq \inf_{u\in L^{2}([0,1]:\mathbb{R}^{d})}\left\{ \frac{1}{2}%
\int_{0}^{1}\left\Vert u(s)\right\Vert ^{2}ds+F\left( \phi
^{A^{1/2}(X^{0})u}\right) \right\}  \label{eq:Laplace Lower}
\end{align}%
and the upper bound 
\begin{align}
& \limsup_{n\rightarrow \infty }-a(n)^{2}\log E\left[ e^{-\frac{1}{a(n)^{2}}%
F(Y^{n})}\right]  \notag \\
& \quad \leq \inf_{u\in L^{2}([0,1]:\mathbb{R}^{d})}\left\{ \frac{1}{2}%
\int_{0}^{1}\left\Vert u(s)\right\Vert ^{2}ds+F\left( \phi
^{A^{1/2}(X^{0})u}\right) \right\} \text{.}  \label{eq: Laplace Upper}
\end{align}%
We will use a tightness and weak convergence result in the proofs of both of
these bounds, but first establish notation used in the rest of the paper. \ 

\begin{construction}
\label{tightnessNotation}\emph{Given a sequence of measures \{}$\eta
^{n}\}_{n\in \mathbb{N}}$\emph{\ with each }$\eta ^{n}\in \mathcal{P}((%
\mathbb{R}^{d})^{n})$\emph{, define the following. \ Let }$(\bar{\upsilon}%
_{0}^{n},\ldots ,\bar{\upsilon}_{n-1}^{n})$\emph{\ be random variables with
distribution }$\eta ^{n}$\emph{, and define }$\{\bar{X}_{i}^{n}\}_{i=0,%
\ldots ,n}$\emph{\ and }$\{\bar{Y}_{i}^{n}\}_{i=0,\ldots ,n}$\emph{\ by (\ref%
{eq:defXbar}) and (\ref{eq:defYbar}). Let }%
\begin{equation*}
\bar{X}^{n}(t)\doteq (i+1-nt)\bar{X}_{i}^{n}+(nt-i)\bar{X}_{i+1}^{n}
\end{equation*}%
\emph{and }%
\begin{equation*}
\bar{Y}^{n}(t)\doteq (i+1-nt)\bar{Y}_{i}^{n}+(nt-i)\bar{Y}_{i+1}^{n}
\end{equation*}%
\emph{for }$t\in \lbrack i/n,i/n+1/n],i=0,\ldots n-1$\emph{\ be their
continuous time linear interpolations. \ Define the conditional means of the
noises }%
\begin{equation*}
w^{n}(t)\doteq \dint_{\mathbb{R}^{d}}y\eta _{i}^{n}(dy)\text{ for }t\in %
\left[ \frac{i}{n},\frac{i+1}{n}\right) ,
\end{equation*}%
\emph{the amplified conditional means}%
\begin{equation*}
\hat{w}^{n}(t)\doteq a(n)\sqrt{n}w^{n}(t),
\end{equation*}%
\emph{and random measures on }$\mathbb{R}^{d}\otimes \lbrack 0,1]$\emph{\ by}%
\begin{equation*}
\hat{\eta}^{n}(dy\otimes dt)\doteq \delta _{\hat{w}^{n}(t)}(dy)dt=\delta
_{a(n)\sqrt{n}w^{n}(t)}(dy)dt.
\end{equation*}
\end{construction}

We will refer to this construction when given $\eta ^{n}$ to identify
associated $\bar{X}^{n},\bar{Y}^{n},\hat{w}^{n}$ and $\hat{\eta}^{n}$. Given 
$\nu \in \mathcal{P}(E_{1}\times E_{2})$, with each $E_{i},i=1,2$ a Polish
space, let $\nu _{2}$ denote the second marginal of $\nu $, and let $\nu
_{1|2}$ denote the conditional distribution on $E_{1}$ given a point in $%
E_{2}$.

\begin{theorem}
\label{th:Tightness}Let $\{\eta ^{n}\}$ be a sequence of measures, each $%
\eta ^{n}\in \mathcal{P}((\mathbb{R}^{d})^{n})$, and define the
corresponding random variables as in Construction \ref{tightnessNotation}. \
Assume that for some $K_{E}<\infty $ 
\begin{equation}
\sup_{n\in \mathbb{N}}\left\{ a(n)^{2}nE\left[ \frac{1}{n}%
\dsum\limits_{i=0}^{n-1}R(\left. \eta _{i}^{n}\right\Vert \mu _{\bar{X}%
_{i}^{n}})\right] \right\} \leq K_{E}\text{.}  \label{cond:relEndBnd}
\end{equation}%
Then $\{(\hat{\eta}^{n},\bar{Y}^{n})\}_{n\in \mathbb{N}}$ is tight in $%
\mathcal{P}(\mathbb{R}^{d}\otimes \lbrack 0,1])\otimes C([0,1]:\mathbb{R}%
^{d})$. Consider a subsequence (keeping the index $n$ for convenience) such
that $\{(\hat{\eta}^{n},\bar{Y}^{n})\}$ converges weakly to $(\hat{\eta},%
\hat{Y})$. \ Then with probability 1 $\hat{\eta}_{2}(dt)$ is Lebesgue
measure and 
\begin{equation}
\hat{Y}(t)=\int_{0}^{t}Db(X^{0}(s))\hat{Y}(s)ds+\int_{0}^{t}\hat{w}(s)ds,
\label{defYhatLim}
\end{equation}%
where 
\begin{equation*}
\hat{w}(t)=\int_{\mathbb{R}^{d}}y\hat{\eta}_{\left. 1\right\vert
2}(dy\left\vert t\right. )\text{.}
\end{equation*}%
In addition,%
\begin{equation}
\liminf_{n\rightarrow \infty }a(n)^{2}nE\left[ \frac{1}{n}%
\sum\limits_{i=0}^{n-1}R(\left. \eta _{i}^{n}\right\Vert \mu _{\bar{X}%
_{i}^{n}})\right] \geq E\left[ \int_{0}^{1}\frac{1}{2}\left\Vert \hat{w}%
(s)\right\Vert _{A^{-1}(X^{0}(s))}^{2}ds\right] .
\label{eqn:lower_bound_costs}
\end{equation}
\end{theorem}

\section{Proof of Theorem \protect\ref{th:Tightness}}

Assume that the bound (\ref{cond:relEndBnd}) holds. \ We will show tightness
of the $\{\hat{\eta}^{n}\}$ measures using the following lemma.

\begin{lemma}
\label{Lemma:RelEntLegIneq}\bigskip Assume Condition \ref{cond:main} and let 
\begin{equation}
L_{c}(x,\beta )=\sup_{\alpha \in \mathbb{R}^{d}}\{\left\langle \alpha ,\beta
\right\rangle -H_{c}(x,\alpha )\}  \label{eq:LT}
\end{equation}%
be the Legendre transform of $H_{c}(x,\cdot )$. \ Then for any $x\in \mathbb{%
R}^{d}$ and $\eta \in \mathcal{P}(\mathbb{R}^{d})$%
\begin{equation*}
R(\left. \eta \right\Vert \mu _{x})\geq L_{c}\left( x,\dint_{\mathbb{R}%
^{d}}y\eta (dy)\right) \text{.}
\end{equation*}
\end{lemma}

\begin{proof}
While the result is likely known we could not locate a proof (see \cite[%
Lemma 6.2.3(f)]{dupell4} for a proof when $H_{c}(x,\alpha )$ is finite for
all $\alpha \in \mathbb{R}^{d}$), and so for completeness provide the
details. \ If $R(\left. \eta \right\Vert \mu _{x})=\infty $ the lemma is
automatically true, so we assume $R(\left. \eta \right\Vert \mu _{x})<\infty 
$. \ Define $\ell (b)\doteq b\log b-b+1$ and note that for $a,b\geq 0$ 
\begin{equation}
ab\leq e^{a}+\ell (b)\text{.}  \label{Product Inequality}
\end{equation}%
From (\ref{MGF Bound}) we have 
\begin{equation*}
\dint_{\mathbb{R}^{d}}e^{\frac{\lambda }{2^{d}}\left\Vert y\right\Vert }\mu
_{x}(dy)\leq 2^{d}e^{dK_{\text{mgf}}}<\infty \text{.}
\end{equation*}%
Therefore 
\begin{align*}
& \dint_{\mathbb{R}^{d}}\frac{\lambda }{2^{d}}\left\Vert y\right\Vert \frac{%
d\eta }{d\mu _{x}}(y)\mu _{x}(dy) \\
& \quad \leq \dint_{\mathbb{R}^{d}}e^{\frac{\lambda }{2^{d}}\left\Vert
y\right\Vert }\mu _{x}(dy)+\dint_{\mathbb{R}^{d}}\ell \left( \frac{d\eta }{%
d\mu }(y)\right) \mu _{x}(dy) \\
& \quad \leq 2^{d}e^{dK_{\text{mgf}}}+R(\left. \eta \right\Vert \mu _{x}),
\end{align*}%
and consequently for any $\alpha \in \mathbb{R}^{d}$ 
\begin{equation}
\dint_{\mathbb{R}^{d}}\left\Vert \alpha \right\Vert \left\Vert y\right\Vert 
\frac{d\eta }{d\mu _{x}}(y)\mu _{x}(dy)\leq \frac{2^{d}\left\Vert \alpha
\right\Vert }{\lambda }\left( 2^{d}e^{dK_{\text{mgf}}}+R(\left. \eta
\right\Vert \mu _{x})\right) <\infty \text{.}  \label{eq:usedforDC}
\end{equation}

Define the bounded, continuous function%
\begin{equation*}
F_{K}(y,\alpha )=\left\{ 
\begin{array}{ll}
\left\langle \alpha ,y\right\rangle \text{ } & \text{if }\left\vert
\left\langle \alpha ,y\right\rangle \right\vert \leq K \\ 
\frac{K\left\langle \alpha ,y\right\rangle }{\left\vert \left\langle \alpha
,y\right\rangle \right\vert } & \text{otherwise,}%
\end{array}%
\right.
\end{equation*}%
and note that (\ref{eq:usedforDC}) and dominated convergence give%
\begin{equation*}
\lim_{K\rightarrow \infty }\dint_{\mathbb{R}^{d}}F_{K}(y,\alpha )\eta
(dy)=\left\langle \alpha ,\dint_{\mathbb{R}^{d}}y\eta (dy)\right\rangle 
\text{.}
\end{equation*}%
In addition, dominated convergence gives 
\begin{equation*}
\lim_{K\rightarrow \infty }\dint_{\{y:\left\langle \alpha ,y\right\rangle
<0\}}e^{F_{K}(y,\alpha )}\mu _{x}(dy)=\dint_{\{y:\left\langle \alpha
,y\right\rangle <0\}}e^{\left\langle \alpha ,y\right\rangle }\mu _{x}(dy)
\end{equation*}%
and monotone convergence gives 
\begin{equation*}
\lim_{K\rightarrow \infty }\dint_{\{y:\left\langle \alpha ,y\right\rangle
\geq 0\}}e^{F_{K}(y,\alpha )}\mu _{x}\left( dy\right)
=\dint_{\{y:\left\langle \alpha ,y\right\rangle \geq 0\}}e^{\left\langle
\alpha ,y\right\rangle }\mu _{x}\left( dy\right) ,
\end{equation*}%
so 
\begin{equation*}
\lim_{K\rightarrow \infty }\log \left( \dint_{\mathbb{R}^{d}}e^{F_{K}(y,%
\alpha )}\mu _{x}\left( dy\right) \right) =H_{c}(x,\alpha )\text{.}
\end{equation*}%
By the Donsker-Varadhan variational formula \cite[Lemma 1.4.3(a)]{dupell4} 
\begin{equation*}
R(\left. \eta \right\Vert \mu _{x})\geq \dint_{\mathbb{R}^{d}}F_{K}(y,\alpha
)\eta (dy)-\log \left( \dint_{\mathbb{R}^{d}}e^{F_{K}(y,\alpha )}\mu
_{x}(dy)\right)
\end{equation*}%
for all $K<\infty $ and $\alpha \in \mathbb{R}^{d}$, and so%
\begin{equation*}
R(\left. \eta \right\Vert \mu _{x})\geq \sup_{\alpha \in \mathbb{R}%
^{d}}\left\{ \left\langle \alpha ,\dint_{\mathbb{R}^{d}}y\eta
(dy)\right\rangle -H_{c}(x,\alpha )\right\} =L_{c}\left( x,\dint_{\mathbb{R}%
^{d}}y\eta (dy)\right) ,
\end{equation*}%
which completes the proof of the lemma.\medskip
\end{proof}

The lemma implies the following theorem, which in turn will give tightness
of $\{\hat{\eta}^{n}\}$.

\begin{theorem}
Assume Condition \ref{cond:main} and (\ref{cond:relEndBnd}). For the
processes $\{w^{n}\}$ obtained in Construction \ref{tightnessNotation} 
\begin{equation*}
\sup_{n\in \mathbb{N}}E\left[ \int_{0}^{1}a(n)\sqrt{n}\left\Vert
w^{n}(s)\right\Vert ds\right] <\infty .
\end{equation*}%
In addition, $\{a(n)\sqrt{n}w^{n}(\cdot )\mathbb{\}}_{n\in \mathbb{N}}$ is
uniformly integrable in the sense that 
\begin{equation*}
\lim_{C\rightarrow \infty }\limsup_{n\rightarrow \infty }E\left[
\dint_{0}^{1}1_{\{a(n)\sqrt{n}\left\Vert w^{n}(s)\right\Vert >C\}}a(n)\sqrt{n%
}\left\Vert w^{n}(s)\right\Vert ds\right] =0\text{.}
\end{equation*}%
\label{th:Bound on Mean}
\end{theorem}

\begin{proof}
We use the following inequality. \ Let $G>0$ satisfy\ \noindent $\lambda
_{DA}\min_{n\in \mathbb{N}}\{a(n)\sqrt{n}\}=\sqrt{G}$ [recall (\ref{a(n)
behavior})] so that $\lambda _{DA}\geq \sqrt{G}/a(n)\sqrt{n}$ for all $n$. \
Define $L_{c}$ by (\ref{eq:LT}). Let $\bar{K}\doteq \lambda
_{DA}K_{DA}+K_{A}/2$. Then with $e_{i}$ denoting the standard unit vectors%
\begin{align*}
& a(n)^{2}nL_{c}(x,\beta ) \\
& \quad =\sup_{\alpha \in \mathbb{R}^{d}}\left[ a(n)\sqrt{n}\left\langle
\alpha ,a(n)\sqrt{n}\beta \right\rangle -a(n)^{2}nH_{c}(x,\alpha )\right] \\
& \quad \geq \pm a(n)\sqrt{n}\left\langle \frac{\sqrt{G}}{a(n)\sqrt{n}}%
e_{i},a(n)\sqrt{n}\beta \right\rangle -a(n)^{2}nH_{c}\left( x,\pm \frac{%
\sqrt{G}}{a(n)\sqrt{n}}e_{i}\right) \\
& \quad \geq \pm \sqrt{G}a(n)\sqrt{n}\beta _{i}-\frac{1}{2}G\left\Vert
A(x)\right\Vert -G\lambda _{DA}K_{DA} \\
& \quad \geq \pm \sqrt{G}a(n)\sqrt{n}\beta _{i}-G\bar{K},
\end{align*}%
where the first inequality follows from making a specific choice of $\alpha $
and the second uses (\ref{eq:H-bound}). Therefore%
\begin{equation}
da(n)^{2}nL_{c}(x,\beta )+dG\bar{K}\geq \sqrt{G}a(n)\sqrt{n}\left\Vert \beta
\right\Vert .  \label{ineq:L-Bound}
\end{equation}%
Using the bound on $L_{c}$ from Lemma \ref{Lemma:RelEntLegIneq} together
with (\ref{cond:relEndBnd}), 
\begin{align}
& d\left( \frac{K_{E}}{\sqrt{G}}+\sqrt{G}\bar{K}\right)  \notag \\
& \quad \geq \frac{da(n)^{2}n}{\sqrt{G}}E\left[ \int_{0}^{1}L_{c}\left( \bar{%
X}^{n}\left( \frac{\left\lfloor ns\right\rfloor }{n}\right) ,w^{n}(s)\right)
ds\right] +d\sqrt{G}\bar{K}  \label{eq:used_in_next} \\
& \quad \geq E\left[ \int_{0}^{1}a(n)\sqrt{n}\left\Vert w^{n}(s)\right\Vert
ds\right] .  \notag
\end{align}

For the uniform integrability, let $C\in (1,\infty )$ be arbitrary and
consider $n$ large enough that 
\begin{equation*}
\min \{\lambda _{DA},1\}\geq \frac{\sqrt{C}}{a(n)\sqrt{n}}\text{.}
\end{equation*}%
Since $\lambda _{DA}\geq 1/a(n)\sqrt{n}$ (which corresponds to using the
estimate above with $G=1$) we have%
\begin{equation*}
E\left[ \int_{0}^{1}a(n)\sqrt{n}\left\Vert w^{n}(s)\right\Vert ds\right]
\leq K^{\ast }\doteq d\left( K_{E}+\frac{1}{2}K_{A}+\lambda
_{DA}K_{DA}\right) .
\end{equation*}%
Therefore 
\begin{equation*}
E\left[ \int_{0}^{1}1_{\{a(n)\sqrt{n}\left\Vert w^{n}(s)\right\Vert >C\}}ds%
\right] \leq \frac{K^{\ast }}{C},
\end{equation*}%
which combined with the estimate (\ref{ineq:L-Bound}) with $G$ replaced by $%
C $ and (\ref{eq:used_in_next}) gives 
\begin{align*}
& \sqrt{C}E\left[ \int_{0}^{1}1_{\{a(n)\sqrt{n}\left\Vert
w^{n}(s)\right\Vert >C\}}a(n)\sqrt{n}\left\Vert w^{n}(s)\right\Vert ds\right]
\\
& \quad \leq E\left[ d\int_{0}^{1}1_{\{a(n)\sqrt{n}\left\Vert
w^{n}(s)\right\Vert >C\}}\left( a(n)^{2}nL_{c}\left( \bar{X}^{n}\left( \frac{%
\left\lfloor ns\right\rfloor }{n}\right) ,w^{n}(s)\right) +C\bar{K}\right) ds%
\right] \\
& \quad \leq da(n)^{2}nE\left[ \int_{0}^{1}L_{c}\left( \bar{X}^{n}\left( 
\frac{\left\lfloor ns\right\rfloor }{n}\right) ,w^{n}(s)\right) ds\right] \\
& \quad \quad +Cd\bar{K}E\left[ \dint_{0}^{1}1_{\{a(n)\sqrt{n}\left\Vert
w^{n}(s)\right\Vert >C\}}ds\right] \\
& \quad \leq K^{\ast }d\left( 1+\bar{K}\right) .
\end{align*}%
We conclude that 
\begin{equation*}
\lim_{C\rightarrow \infty }\limsup_{n\rightarrow \infty }E\left[
\int_{0}^{1}1_{\{a(n)\sqrt{n}\left\Vert w^{n}(s)\right\Vert >C\}}a(n)\sqrt{n}%
\left\Vert w^{n}(s)\right\Vert ds\right] =0,
\end{equation*}%
which is the claimed uniform integrability.
\end{proof}

\bigskip

We continue with the proof of Theorem \ref{th:Tightness}. Note that $%
g(y,t)=\left\Vert y\right\Vert $ is a tightness function on $\mathbb{R}%
^{d}\otimes \lbrack 0,1]$, so by \cite[Theorem A.3.17]{dupell4}%
\begin{equation*}
G(\eta )=\int_{\mathbb{R}^{d}\otimes \lbrack 0,1]}\left\Vert y\right\Vert
\eta (dy\otimes dt)
\end{equation*}%
is a tightness function on $\mathcal{P(}\mathbb{R}^{d}\otimes \lbrack 0,1])$
and 
\begin{equation*}
\bar{G}(\gamma )=\int_{\mathcal{P}\left( \mathbb{R}^{d}\otimes \lbrack
0,1]\right) }\int_{\mathbb{R}^{d}\otimes \lbrack 0,1]}\left\Vert
y\right\Vert \eta (dy\otimes dt)\gamma (d\eta )
\end{equation*}%
is a tightness function on $\mathcal{P(P(}\mathbb{R}^{d}\otimes \lbrack
0,1]))$. \ Since 
\begin{equation*}
\sup_{n\in \mathbb{N}}EG(\hat{\eta}^{n})=\sup_{n\in \mathbb{N}}E\left[ \dint
\left\Vert y\right\Vert \hat{\eta}^{n}(dy\otimes dt)\right] =\sup_{n\in 
\mathbb{N}}E\left[ \int_{0}^{1}a(n)\sqrt{n}\left\Vert w^{n}(s)\right\Vert ds%
\right] <\infty \text{,}
\end{equation*}%
$\{\hat{\eta}^{n}\}$ is tight and consequently there is a subsequence of $\{%
\hat{\eta}^{n}\}$ which converges weakly. To simplify notation we retain $n$
as the index of this convergent subsequence, and denote the weak limit of $\{%
\hat{\eta}^{n}\}$ by $\hat{\eta}$. \ Note that for all $n$ the second
marginal of $\hat{\eta}^{n}(dy\otimes dt)$, which we denote by $\hat{\eta}%
_{2}^{n}(dt)$, is Lebesgue measure, and therefore $\hat{\eta}_{2}(dt)$ is
Lebesgue measure with probability 1. \ 

Our aim is to show that $\bar{Y}^{n}(t)\rightarrow \hat{Y}(t)$ weakly in $%
C([0,1]:\mathbb{R}^{d})$, where $\hat{Y}(t)$ is given by (\ref{defYhatLim})
in terms of the weak limit $\hat{\eta}$. \ To achieve this we introduce the
following processes which serve as intermediate steps. \ Let $\check{Y}%
_{0}^{n}=0$ and%
\begin{equation}
\check{Y}_{i+1}^{n}=\check{Y}_{i}^{n}+\frac{a(n)}{\sqrt{n}}\left( b\left(
X_{i}^{n,0}+\frac{1}{a(n)\sqrt{n}}\check{Y}_{i}^{n}\right) -b\left(
X_{i}^{n,0}\right) \right) +\frac{a(n)}{\sqrt{n}}w^{n}\left( \frac{i}{n}%
\right) ,  \notag
\end{equation}%
together with its continuous time linear interpolation defined for $t\in
\lbrack i/n,i/n+1/n]$ by 
\begin{equation*}
\check{Y}^{n}(t)=(i+1-nt)\check{Y}_{i}^{n}+(nt-i)\check{Y}_{i+1}^{n}\text{.}
\end{equation*}%
Also let 
\begin{equation}
\hat{Y}^{n}(t)=\int_{0}^{t}Db\left( X^{0}(s)\right) \hat{Y}%
^{n}(s)ds+\int_{0}^{t}\hat{w}^{n}(s)ds  \label{defYhatSeq}
\end{equation}%
where 
\begin{equation*}
\hat{w}^{n}(t)=\int_{\mathbb{R}^{d}}y\hat{\eta}_{\left. 1\right\vert
2}^{n}(dy\left\vert t\right. )
\end{equation*}%
as in Construction \ref{tightnessNotation}. \ These are both random
variables taking values in $C([0,1]:\mathbb{R}^{d})$. \ Note that $\bar{Y}%
^{n}$ differs from $\check{Y}^{n}$ because $\bar{Y}^{n}$ is driven by the
actual noises and $\check{Y}^{n}$ is driven by their conditional means. \
While the driving terms of $\hat{Y}^{n}$ and $\check{Y}^{n}$ are the same
[recall that $a(n)\sqrt{n}w^{n}(t)=\hat{w}^{n}(t)$], they differ in that $%
\check{Y}^{n}$ is still a linear interpolation of a discrete time process
whereas $\hat{Y}^{n}$ satisfies an ODE. \ The goal is to show that along the
subsequence where $\hat{\eta}^{n}\rightarrow \hat{\eta}$ weakly 
\begin{equation*}
\bar{Y}^{n}-\check{Y}^{n}\rightarrow 0,\check{Y}^{n}-\hat{Y}^{n}\rightarrow
0,\text{ and }\hat{Y}^{n}\rightarrow \hat{Y}
\end{equation*}%
in $C([0,1]:\mathbb{R}^{d})$, all in distribution. \ To show $\hat{Y}%
^{n}\rightarrow \hat{Y}$ we show that $\{\hat{Y}^{n}\}$ is tight in $C([0,1]:%
\mathbb{R}^{d})$ and use the mapping defined by (\ref{defYhatSeq}) from $%
\int_{0}^{\cdot }\hat{w}^{n}$ to $\hat{Y}^{n}$. \ Recall that $\sup_{x\in 
\mathbb{R}^{d}}\left\Vert Db(x)\right\Vert \leq K_{b}$. \ The following
lemma is an easy consequence of Gronwall's inequality.

\begin{lemma}
\ Let $u\in L^{1}([0,1]:\mathbb{R}^{d})$ be arbitrary and $\phi ^{u}$ be
defined as in (\ref{phiDefined}). Then for $0\leq s\leq t\leq 1$ 
\begin{equation*}
\left\Vert \phi ^{u}(t)-\phi ^{u}(s)\right\Vert \leq
(t-s)K_{b}e^{K_{b}}\int_{0}^{1}\left\Vert u(r)\right\Vert
dr+\dint_{s}^{t}\left\Vert u(r)\right\Vert dr\text{.}
\end{equation*}
\end{lemma}

With this lemma and the uniform integrability of $\{\hat{\eta}^{n}\}$ given
in Theorem \ref{th:Bound on Mean}, tightness follows.

\begin{lemma}
\label{lem:tight_Y_W}Assume Condition \ref{cond:main} and (\ref%
{cond:relEndBnd}). The sequence $\{\hat{Y}^{n}\}$ defined in (\ref%
{defYhatSeq}) in terms of the measures $\{\eta ^{n}\}$ via Construction \ref%
{tightnessNotation}\ is tight in $C([0,1]:\mathbb{R}^{d})$, as is $%
\{\int_{0}^{\cdot }\hat{w}^{n}ds\}$.
\end{lemma}

\begin{proof}
It suffices to show that for any $\varepsilon >0$ there is $\delta >0$ such
that%
\begin{equation*}
\limsup_{n\rightarrow \infty }P\left( \sup_{\left\vert s-t\right\vert \leq
\delta }\left\Vert \hat{Y}^{n}(t)-\hat{Y}^{n}(s)\right\Vert >\varepsilon
\right) <\varepsilon \text{.}
\end{equation*}%
Define 
\begin{eqnarray*}
T(C) &\doteq &\limsup_{n\rightarrow \infty }E\left[ \dint_{0}^{1}1_{\{\left%
\Vert \hat{w}^{n}\left( t\right) \right\Vert >C\}}\left\Vert \hat{w}%
^{n}(t)\right\Vert dt\right] \\
&=&\limsup_{n\rightarrow \infty }E\left[ \int_{\{\left\Vert y\right\Vert
>C\}}\left\Vert y\right\Vert \hat{\eta}^{n}(dy\otimes dt)\right] .
\end{eqnarray*}%
By Theorem \ref{th:Bound on Mean} $T(C)\rightarrow 0$ as $C\rightarrow
\infty $. Define also $K_{\eta }=\sup_{n\in \mathbb{N}}E\int_{0}^{1}\left%
\Vert \hat{w}^{n}(t)\right\Vert dt$, which is finite by Theorem \ref%
{th:Bound on Mean}. Let $\varepsilon >0$ be arbitrary. \ Then for any $s<t$
satisfying $t-s\leq \delta $ the previous lemma implies%
\begin{equation*}
\left\Vert \hat{Y}^{n}(t)-\hat{Y}^{n}(s)\right\Vert \leq \delta
K_{b}e^{K_{b}}\int_{0}^{1}\left\Vert \hat{w}^{n}(r)\right\Vert
dr+\dint_{s}^{t}\left\Vert \hat{w}^{n}(r)\right\Vert dr.
\end{equation*}%
Since 
\begin{equation*}
\dint_{s}^{t}\left\Vert \hat{w}^{n}(r)\right\Vert dr\leq C\delta
+\dint_{0}^{1}1_{\{\left\Vert \hat{w}^{n}(r)\right\Vert >C\}}\left\Vert \hat{%
w}^{n}(r)\right\Vert dr,
\end{equation*}%
it follows that 
\begin{align*}
\left\Vert \hat{Y}^{n}(t)-\hat{Y}^{n}(s)\right\Vert & \leq \delta \left(
C+K_{b}e^{K_{b}}\int_{0}^{1}\left\Vert \hat{w}^{n}(r)\right\Vert dr\right) \\
& \quad +\dint_{0}^{1}1_{\{\left\Vert \hat{w}^{n}(r)\right\Vert
>C\}}\left\Vert \hat{w}^{n}(r)\right\Vert dr\text{.}
\end{align*}%
Hence by Markov's inequality%
\begin{align*}
& \limsup_{n\rightarrow \infty }P\left( \sup_{\left\vert s-t\right\vert \leq
\delta }\left\Vert \hat{Y}^{n}(t)-\hat{Y}^{n}(s)\right\Vert >\varepsilon
\right) \\
& \quad \leq \frac{\delta }{\varepsilon }\limsup_{n\rightarrow \infty }E%
\left[ \left( C+K_{b}e^{K_{b}}\int_{0}^{1}\left\Vert \hat{w}%
^{n}(r)\right\Vert dr\right) \right] \\
& \qquad +\frac{1}{\varepsilon }\limsup_{n\rightarrow \infty }E\left[
\dint_{0}^{1}1_{\{\left\Vert \hat{w}^{n}(r)\right\Vert >C\}}\left\Vert \hat{w%
}^{n}(r)\right\Vert dr\right] \\
& \quad \leq \frac{\delta }{\varepsilon }(C+K_{b}e^{K_{b}}K_{\eta })+\frac{1%
}{\varepsilon }T(C).
\end{align*}%
Choose $C<\infty $ such that $T(C)<\varepsilon ^{2}/2$ and then choose $%
\delta >0$ so that the $\delta (C+K_{b}e^{K_{b}}K_{\eta })<\varepsilon
^{2}/2 $. This shows the tightness of $\{\hat{Y}^{n}\}$. The tightness of $%
\{\int_{0}^{\cdot }\hat{w}^{n}ds\}$ is simpler, and follows from the bound 
\begin{equation*}
\limsup_{n\rightarrow \infty }P\left( \sup_{\left\vert s-t\right\vert \leq
\delta }\dint_{s}^{t}\left\Vert \hat{w}^{n}(r)\right\Vert dr>\varepsilon
\right) \leq \delta \frac{C}{\varepsilon }+\frac{1}{\varepsilon }T(C).
\end{equation*}
\end{proof}

\medskip\ \ We still need to show that $\hat{Y}^{n}$ converges to $\hat{Y}$.
\ This also relies on the uniform integrability given by Theorem \ref%
{th:Bound on Mean}.

\begin{lemma}
\label{lem:limit_Y}Assume Condition \ref{cond:main} and (\ref{cond:relEndBnd}%
). \ Let the sequence $\{\hat{Y}^{n}\left( t\right) \}$ be defined by (\ref%
{defYhatSeq}), let $\hat{Y}(t)$ be defined by (\ref{defYhatLim}), and
consider a convergent subsequence $\{(\hat{Y}^{n},\hat{\eta}^{n})\}$ with
limit $(\hat{Y}^{\ast },\hat{\eta})$. Then w.p.1 $\hat{Y}^{\ast }=\hat{Y}$.
\end{lemma}

\begin{proof}
\ We can write 
\begin{equation*}
\hat{Y}^{n}(t)=\int_{0}^{t}Db(X^{0}(s))\hat{Y}^{n}(s)ds+\int_{0}^{t}\int_{%
\mathbb{R}^{d}}y\hat{\eta}^{n}(dy\otimes ds).
\end{equation*}%
Using the uniform integrability proved in Theorem \ref{th:Bound on Mean} and
that $\hat{\eta}_{2}$ is Lebesgue measure w.p.1, sending $n\rightarrow
\infty $ and using the definition of $\hat{w}$ gives%
\begin{eqnarray*}
\hat{Y}^{\ast }(t) &=&\int_{0}^{t}Db(X^{0}(s))\hat{Y}^{\ast
}(s)ds+\int_{0}^{t}\int_{\mathbb{R}^{d}}y\hat{\eta}(dy\otimes ds) \\
&=&\int_{0}^{t}Db(X^{0}(s))\hat{Y}^{\ast }(s)ds+\int_{0}^{t}\hat{w}(s)ds.
\end{eqnarray*}%
By uniqueness of the solution, $\hat{Y}^{\ast }=\hat{Y}$ follows.
\end{proof}

\medskip It remains to show $\bar{Y}^{n}-\check{Y}^{n}\rightarrow 0$ and $%
\check{Y}^{n}-\hat{Y}^{n}\rightarrow 0$. \ We begin with $\bar{Y}^{n}-\check{%
Y}^{n}\rightarrow 0$. \ Recall that the difference between $\bar{Y}^{n}$ and 
$\check{Y}^{n}$ is that the first is driven by the actual noises and the
second is driven by their conditional means. \ The following theorem is a
law of large numbers type result for the difference between the noises and
their conditional means, and is the most complicated part of the analysis.

\begin{theorem}
\label{th:Wconv_in_prob}Assume Condition \ref{cond:main} and (\ref%
{cond:relEndBnd}). Consider the sequence $\{\bar{\upsilon}%
_{i}^{n}\}_{i=0,\ldots ,n-1}$ of controlled noises and $\{w^{n}(i/n)\}_{i=0,%
\ldots ,n-1}$ of means of the controlled noises as in Construction \ref%
{tightnessNotation}. For $i\in \{1,\ldots ,n\}$ let%
\begin{equation*}
W_{i}^{n}\doteq \frac{1}{n}\dsum\limits_{j=0}^{i-1}a(n)\sqrt{n}\left( \bar{%
\upsilon}_{i}^{n}-w^{n}\left( i/n\right) \right) \text{.}
\end{equation*}%
Then for any $\delta >0$ 
\begin{equation*}
\lim_{n\rightarrow \infty }P\left[ \max_{i\in \{1,\ldots ,n\}}\left\Vert
W_{i}^{n}\right\Vert \geq \delta \right] =0\text{.}
\end{equation*}
\end{theorem}

\begin{proof}
According to (\ref{cond:relEndBnd})%
\begin{equation*}
\frac{1}{n}\sum\limits_{i=0}^{n-1}E[R(\left. \eta _{i}^{n}\right\Vert \mu _{%
\bar{X}_{i}^{n}})]\leq \frac{K_{E}}{a^{2}(n)n}.
\end{equation*}%
Because of this the (random) Radon-Nikodym derivatives 
\begin{equation*}
f_{i}^{n}(y)=\frac{d\eta _{i}^{n}}{d\mu _{\bar{X}_{i}^{n}}}(y)
\end{equation*}%
are well defined and can be selected in a measurable way. We will control
the magnitude of the noise when the Radon-Nikodym derivative is large by
bounding 
\begin{equation*}
\frac{1}{n}\sum\limits_{i=0}^{n-1}E[1_{\{f_{i}^{n}(\bar{\upsilon}%
_{i}^{n})\geq r\}}\left\Vert \bar{\upsilon}_{i}^{n}\right\Vert ]
\end{equation*}%
for large $r$. \ 

From the bound on the moment generating function (\ref{MGF Bound}), 
\begin{equation*}
\sup_{x\in \mathbb{R}^{d}}\dint_{\mathbb{R}^{d}}e^{\frac{\lambda }{2^{d}}%
\left\Vert y\right\Vert }\mu _{x}(dy)\leq 2^{d}e^{dK_{\text{mgf}}}\text{.}
\end{equation*}%
Let $\sigma =\min \{\lambda /2^{d+1},1\}$ and recall the definition $\ell
(b)\doteq b\log b-b+1$. Then 
\begin{equation*}
\bigskip \frac{1}{n}\sum\limits_{i=0}^{n-1}E\left[ 1_{\{f_{i}^{n}(\bar{%
\upsilon}_{i}^{n})\geq r\}}\left\Vert \bar{\upsilon}_{i}^{n}\right\Vert %
\right] =\frac{1}{n}\sum\limits_{i=0}^{n-1}E\left[ \int_{\{y:f_{i}^{n}(y)%
\geq r\}}\left\Vert y\right\Vert f_{i}^{n}(y)\mu _{\bar{X}_{i}^{n}}(dy)%
\right]
\end{equation*}%
and the bound $ab\leq e^{a}+\ell (b)$ for $a,b\geq 0$ with $a=\sigma
\left\Vert y\right\Vert $ and $b=f_{i}^{n}(y)$ gives that for all $i$ 
\begin{align*}
& E\left[ \int_{\{y:f_{i}^{n}(y)\geq r\}}\left\Vert y\right\Vert
f_{i}^{n}(y)\mu _{\bar{X}_{i}^{n}}(dy)\right] \\
& \quad \leq \frac{1}{\sigma }E\left[ \int_{\{y:f_{i}^{n}(y)\geq
r\}}e^{\sigma \left\Vert y\right\Vert }\mu _{\bar{X}_{i}^{n}}(dy)\right] +%
\frac{1}{\sigma }E\left[ \int_{\{y:f_{i}^{n}(y)\geq r\}}\ell
(f_{i}^{n}(y))\mu _{\bar{X}_{i}^{n}}(dy)\right] .
\end{align*}%
Since $\ell (b)\geq 0$ for all $b\geq 0$ 
\begin{align*}
E\left[ \int_{\left\{ y:f_{i}^{n}\left( y\right) \geq r\right\} }\ell \left(
f_{i}^{n}\left( y\right) \right) \mu _{\bar{X}_{i}^{n}}\left( dy\right) %
\right] & \leq E\left[ \int_{\mathbb{R}^{d}}\ell (f_{i}^{n}(y))\mu _{\bar{X}%
_{i}^{n}}(dy)\right] \\
& =E[R(\left. \eta _{i}^{n}\right\Vert \mu _{\bar{X}_{i}^{n}})],
\end{align*}%
and by Hölder's inequality 
\begin{align*}
& E\left[ \int_{\{y:f_{i}^{n}(y)\geq r\}}e^{\sigma \left\Vert y\right\Vert
}\mu _{\bar{X}_{i}^{n}}(dy)\right] \\
& \quad \leq E\left[ \left( \dint_{\mathbb{R}^{d}}1_{\{f_{i}^{n}(y)\geq
r\}}\mu _{\bar{X}_{i}^{n}}(dy)\right) ^{\frac{1}{2}}\left( \dint_{\mathbb{R}%
^{d}}e^{2\sigma \left\Vert y\right\Vert }\mu _{\bar{X}_{i}^{n}}(dy)\right) ^{%
\frac{1}{2}}\right] \\
& \quad =E\left[ \mu _{\bar{X}_{i}^{n}}(\{y:f_{i}^{n}(y)\geq r\})^{\frac{1}{2%
}}\right] \left( 2^{d}e^{dK_{\text{mgf}}}\right) ^{\frac{1}{2}}.
\end{align*}%
In addition Markov's inequality gives for $r\geq e^{-1}$ 
\begin{equation*}
\mu _{\bar{X}_{i}^{n}}(\{y:f_{i}^{n}(y)\geq r\})\leq \frac{1}{r\log r}\dint
\log (f_{i}^{n}(y))f_{i}^{n}(y)\mu _{\bar{X}_{i}^{n}}(dy)=\frac{R(\left.
\eta _{i}^{n}\right\Vert \mu _{\bar{X}_{i}^{n}})}{r\log r}\text{.}
\end{equation*}%
Therefore 
\begin{align*}
& \frac{1}{n}\sum\limits_{i=0}^{n-1}E\left[ \int_{\{f_{i}^{n}(y)\geq
r\}}\left\Vert y\right\Vert f_{i}^{n}(y)\mu _{\bar{X}_{i}^{n}}(dy)\right] \\
& \quad \leq \frac{1}{\sigma }\left( 2^{d}e^{dK_{\text{mgf}}}\right) ^{\frac{%
1}{2}}\frac{1}{n}\sum\limits_{i=0}^{n-1}E\left[ \left( \frac{R(\left. \eta
_{i}^{n}\right\Vert \mu _{\bar{X}_{i}^{n}})}{r\log r}\right) ^{\frac{1}{2}}%
\right] +\frac{1}{\sigma }\frac{1}{n}\sum\limits_{i=0}^{n-1}E[R(\left. \eta
_{i}^{n}\right\Vert \mu _{\bar{X}_{i}^{n}})].
\end{align*}%
Since by Jensen's inequality 
\begin{equation*}
\frac{1}{n}\sum\limits_{i=0}^{n-1}E\left[ \left( \frac{R(\left. \eta
_{i}^{n}\right\Vert \mu _{\bar{X}_{i}^{n}})}{r\log r}\right) ^{\frac{1}{2}}%
\right] \leq \left( \frac{1}{r\log r}\right) ^{\frac{1}{2}}\left( \frac{1}{n}%
\sum\limits_{i=0}^{n-1}E[R(\left. \eta _{i}^{n}\right\Vert \mu _{\bar{X}%
_{i}^{n}})]\right) ^{\frac{1}{2}},
\end{equation*}%
we obtain the overall bound%
\begin{align}
& \frac{1}{n}\sum\limits_{i=0}^{n-1}E\left[ 1_{\{f_{i}^{n}(\bar{\upsilon}%
_{i}^{n})\geq r\}}\left\Vert \bar{\upsilon}_{i}^{n}\right\Vert \right] 
\notag \\
& \quad \leq \frac{1}{\sigma }\left( 2^{d}e^{dK_{\text{mgf}}}\right) ^{\frac{%
1}{2}}\left( \frac{1}{r\log r}\right) ^{\frac{1}{2}}\left( \frac{1}{n}%
\sum\limits_{i=0}^{n-1}E[R(\left. \eta _{i}^{n}\right\Vert \mu _{\bar{X}%
_{i}^{n}})]\right) ^{\frac{1}{2}}  \notag \\
& \qquad +\frac{1}{\sigma }\frac{1}{n}\sum\limits_{i=0}^{n-1}E[R(\left. \eta
_{i}^{n}\right\Vert \mu _{\bar{X}_{i}^{n}})]  \notag \\
& \quad \leq \frac{1}{\sigma }\frac{K_{E}^{\frac{1}{2}}}{a(n)\sqrt{n}}\left(
2^{d}e^{dK_{\text{mgf}}}\right) ^{\frac{1}{2}}\left( \frac{1}{r\log r}%
\right) ^{\frac{1}{2}}+\frac{1}{\sigma }\frac{K_{E}}{a(n)^{2}n}\text{.}
\label{ineq:bndNoiseLargeRNderiv}
\end{align}

Using this result we can complete the proof. \ Define%
\begin{equation*}
\xi _{i}^{n,r}\doteq \left\{ 
\begin{array}{cc}
\bar{v}_{i}^{n} & \text{if }f_{i}^{n}(\bar{\upsilon}_{i}^{n})<r \\ 
0 & \text{otherwise.}%
\end{array}%
\right.
\end{equation*}%
\ For any for any $\delta >0$%
\begin{align*}
& P\left\{ \max_{k=0,...,n-1}\left\Vert \frac{1}{n}\sum\limits_{i=0}^{k}a(n)%
\sqrt{n}\left( \bar{\upsilon}_{i}^{n}-w^{n}\left( \frac{i}{n}\right) \right)
\right\Vert \geq 3\delta \right\} \\
& \quad \leq P\left\{ \max_{k=0,...,n-1}\left\Vert \frac{1}{n}%
\sum\limits_{i=0}^{k}a(n)\sqrt{n}(\bar{\upsilon}_{i}^{n}-\xi
_{i}^{n,r})\right\Vert \geq \delta \right\} \\
& \qquad +P\left\{ \max_{k=0,...,n-1}\left\Vert \frac{1}{n}%
\sum\limits_{i=0}^{k}a(n)\sqrt{n}\left( \xi
_{i}^{n,r}-\int_{\{y:f_{i}^{n}(y)<r\}}y\eta _{i}^{n}(dy)\right) \right\Vert
\geq \delta \right\} \\
& \qquad +P\left\{ \max_{k=0,...,n-1}\left\Vert \frac{1}{n}%
\sum\limits_{i=0}^{k}a(n)\sqrt{n}\left( w^{n}\left( \frac{i}{n}\right)
-\int_{\{y:f_{i}^{n}(y)<r\}}y\eta _{i}^{n}(dy)\right) \right\Vert \geq
\delta \right\} \text{.}
\end{align*}%
The first term satisfies 
\begin{align*}
& P\left\{ \max_{k=0,...,n-1}\left\Vert \frac{1}{n}\sum\limits_{i=0}^{k}a(n)%
\sqrt{n}(\bar{\upsilon}_{i}^{n}-\xi _{i}^{n,r})\right\Vert \geq \delta
\right\} \\
& \quad \leq \frac{1}{\delta }a(n)\sqrt{n}\frac{1}{n}\sum\limits_{i=0}^{n-1}E%
\left[ \left\Vert \bar{\upsilon}_{i}^{n}-\xi _{i}^{n,r}\right\Vert \right] \\
& \quad =\frac{1}{\delta }a(n)\sqrt{n}\frac{1}{n}\sum\limits_{i=0}^{n-1}E%
\left[ 1_{\{f_{i}^{n}(\bar{\upsilon}_{i}^{n})\geq r\}}\left\Vert \bar{%
\upsilon}_{i}^{n}\right\Vert \right] \text{.}
\end{align*}%
The second term is a submartingale so by Doob's submartingale inequality 
\begin{align*}
& P\left\{ \max_{k=0,...,n-1}\left\Vert \frac{1}{n}\sum\limits_{i=0}^{k}a(n)%
\sqrt{n}\left( \xi _{i}^{n,r}-\int_{\{y:f_{i}^{n}(y)<r\}}y\eta
_{i}^{n}(dy)\right) \right\Vert \geq \delta \right\} \\
& \quad \leq \frac{1}{\delta ^{2}}E\left[ \left\Vert \frac{1}{n}%
\sum\limits_{i=0}^{n-1}a(n)\sqrt{n}\left( \xi
_{i}^{n,r}-\int_{\{y:f_{i}^{n}(y)<r\}}y\eta _{i}^{n}(dy)\right) \right\Vert
^{2}\right] \\
& \quad =\frac{1}{\delta ^{2}}\frac{a(n)^{2}}{n}\sum\limits_{i=0}^{n-1}E%
\left[ \left\Vert \left( \xi _{i}^{n,k}-\int_{\{y:f_{i}^{n}(y)<r\}}y\eta
_{i}^{n}(dy)\right) \right\Vert ^{2}\right] \\
& \quad \leq \frac{1}{\delta ^{2}}\frac{a(n)^{2}}{n}\sum\limits_{i=0}^{n-1}E%
\left[ \left\Vert \xi _{i}^{n,k}\right\Vert ^{2}\right] \\
& \quad =\frac{1}{\delta ^{2}}\frac{a(n)^{2}}{n}\sum\limits_{i=0}^{n-1}E%
\left[ \int_{\{y:f_{i}^{n}(y)<r\}}\left\Vert y\right\Vert
^{2}f_{i}^{n}(y)\mu _{\bar{X}_{i}^{n}}(dy)\right] \\
& \quad \leq \frac{r}{\delta ^{2}}\frac{a(n)^{2}}{n}\sum\limits_{i=0}^{n-1}E%
\left[ \int_{\mathbb{R}^{d}}\left\Vert y\right\Vert ^{2}\mu _{\bar{X}%
_{i}^{n}}(dy)\right] \\
& \quad \leq \frac{r}{\delta ^{2}}a(n)^{2}K_{\mu ,2},
\end{align*}%
where 
\begin{equation*}
K_{\mu ,2}=\sup_{x\in \mathbb{R}^{d}}\dint_{\mathbb{R}^{d}}\left\Vert
y\right\Vert ^{2}\mu _{x}(dy)<\infty ,
\end{equation*}%
and the finiteness is due to (\ref{MGF Bound}). \ We can use Jensen's
inequality with the third term and get the same bound that was shown for the
first. \ We have 
\begin{align*}
& P\left\{ \max_{k=0,...,n-1}\left\Vert \frac{1}{n}\sum\limits_{i=0}^{k}a(n)%
\sqrt{n}\left( w^{n}\left( \frac{i}{n}\right)
-\int_{\{y:f_{i}^{n}(y)<r\}}y\eta _{i}^{n}(dy)\right) \right\Vert \geq
\delta \right\} \\
& \quad \leq \frac{1}{\delta }a(n)\sqrt{n}\frac{1}{n}\sum\limits_{i=0}^{n-1}E%
\left[ \left\Vert \left( w^{n}\left( \frac{i}{n}\right)
-\int_{\{y:f_{i}^{n}(y)<r\}}y\eta _{i}^{n}(dy)\right) \right\Vert \right] \\
& \quad =\frac{1}{\delta }a(n)\sqrt{n}\frac{1}{n}\sum\limits_{i=0}^{n-1}E%
\left[ \left\Vert \int_{\{y:f_{i}^{n}(y)\geq r\}}y\eta
_{i}^{n}(dy)\right\Vert \right] \\
& \quad \leq \frac{1}{\delta }a(n)\sqrt{n}\frac{1}{n}\sum\limits_{i=0}^{n-1}E%
\left[ \int_{\{y:f_{i}^{n}(y)\geq r\}}\left\Vert y\right\Vert \eta
_{i}^{n}(dy)\right] \\
& \quad =\frac{1}{\delta }a(n)\sqrt{n}\frac{1}{n}\sum\limits_{i=0}^{n-1}E%
\left[ 1_{\{f_{i}^{n}(\bar{\upsilon}_{i}^{n})\geq r\}}\left\Vert \bar{%
\upsilon}_{i}^{n}\right\Vert \right] \text{.}
\end{align*}%
Combining the bounds for these three terms with (\ref%
{ineq:bndNoiseLargeRNderiv}) gives%
\begin{align*}
& P\left\{ \max_{k=0,...,n-1}\left\Vert \frac{1}{n}\sum\limits_{i=0}^{k}a(n)%
\sqrt{n}\left( \bar{\upsilon}_{i}^{n}-w^{n}\left( \frac{i}{n}\right) \right)
\right\Vert \geq 3\delta \right\} \\
& \quad \leq \frac{2}{\delta }a(n)\sqrt{n}\frac{1}{n}\sum\limits_{i=0}^{n-1}E%
\left[ 1_{\{f_{i}^{n}(\bar{\upsilon}_{i}^{n})\geq r]}\left\Vert \bar{\upsilon%
}_{i}^{n}\right\Vert \right] +\frac{r}{\delta ^{2}}a(n)^{2}K_{\mu ,2} \\
& \quad \leq \frac{2}{\sigma \delta }K_{E}^{\frac{1}{2}}\left( 2^{d}e^{dK_{%
\text{mgf}}}\right) ^{\frac{1}{2}}\left( \frac{1}{r\log r}\right) ^{\frac{1}{%
2}}+\frac{2}{\sigma \delta }\frac{K_{E}}{a(n)\sqrt{n}}+a(n)^{2}\frac{r}{%
\delta ^{2}}K_{\mu ,2}.
\end{align*}%
Choosing $r=1/a(n)$ and using $a(n)\rightarrow 0,a(n)\sqrt{n}\rightarrow
\infty $ gives 
\begin{equation*}
P\left\{ \max_{k=0,...,n-1}\left\Vert \frac{1}{n}\sum\limits_{i=0}^{k}a(n)%
\sqrt{n}\left( \bar{\upsilon}_{i}^{n}-w^{n}\left( \frac{i}{n}\right) \right)
\right\Vert \geq 3\delta \right\} \rightarrow 0
\end{equation*}%
as $n\rightarrow \infty $, which completes the proof.
\end{proof}

\medskip This theorem, combined with the following discrete version of
Gronwall's inequality, will allow us to prove $\bar{Y}^{n}-\check{Y}%
^{n}\rightarrow 0$.

\begin{lemma}
\label{lemma:Discrete Gronwall}\bigskip \bigskip If $\{a_{n}\}$, $\{b_{n}\}$%
, and $\{c_{n}\}$ are nonnegative sequences defined for $n=0,1,\ldots $ and
satisfying 
\begin{equation*}
a_{n}\leq c_{n}+\dsum\limits_{k=0}^{n-1}b_{k}a_{k},
\end{equation*}%
then 
\begin{equation*}
a_{n}\leq c_{n}+\dsum\limits_{k=0}^{n-1}b_{k}c_{k}\exp \left\{
\dsum\limits_{i=k+1}^{n-1}b_{i}\right\} \text{.}
\end{equation*}
\end{lemma}

\begin{theorem}
\label{th:hat-bar}Under the conditions of Theorem \ref{th:Wconv_in_prob} $%
\check{Y}^{n}-\bar{Y}^{n}\rightarrow 0$ in probability.
\end{theorem}

\begin{proof}
Recall that%
\begin{equation*}
\bar{Y}_{k}^{n}=\dsum\limits_{i=0}^{k-1}\frac{a(n)}{\sqrt{n}}\left( b\left(
X_{i}^{n,0}+\frac{1}{a(n)\sqrt{n}}\bar{Y}_{i}^{n}\right) -b\left(
X_{i}^{n,0}\right) \right) +\dsum\limits_{i=0}^{k-1}\frac{a(n)}{\sqrt{n}}%
\bar{\upsilon}_{i}^{n}
\end{equation*}%
and%
\begin{equation*}
\check{Y}_{k}^{n}=\dsum\limits_{i=0}^{k-1}\frac{a(n)}{\sqrt{n}}\left(
b\left( X_{i}^{n,0}+\frac{1}{a(n)\sqrt{n}}\check{Y}_{i}^{n}\right) -b\left(
X_{i}^{n,0}\right) \right) +\dsum\limits_{i=0}^{k-1}\frac{a(n)}{\sqrt{n}}%
w^{n}\left( \frac{i}{n}\right) ,
\end{equation*}%
so with $W_{k}^{n}$ defined as in Theorem \ref{th:Wconv_in_prob} 
\begin{equation*}
\left\Vert \bar{Y}_{k}^{n}-\check{Y}_{k}^{n}\right\Vert \leq \left\Vert
W_{k}^{n}\right\Vert +\dsum\limits_{i=0}^{k-1}\frac{K_{b}}{n}\left\Vert \bar{%
Y}_{i}^{n}-\check{Y}_{i}^{n}\right\Vert \text{.}
\end{equation*}%
Using Lemma \ref{lemma:Discrete Gronwall} gives 
\begin{align*}
\left\Vert \bar{Y}_{k}^{n}-\check{Y}_{k}^{n}\right\Vert & \leq \left\Vert
W_{k}^{n}\right\Vert +\dsum\limits_{i=0}^{k-1}\left\Vert
W_{i}^{n}\right\Vert \frac{K_{b}}{n}\exp \left\{ \frac{K_{b}}{n}%
(k-i-1)\right\} \\
& \leq (1+K_{b}e^{K_{b}})\max_{i\in \{1,\ldots ,k\}}\{\left\Vert
W_{i}^{n}\right\Vert \}
\end{align*}%
so 
\begin{equation*}
\max_{i\in \{1,\ldots ,n\}}\left\{ \left\Vert \bar{Y}_{i}^{n}-\check{Y}%
_{i}^{n}\right\Vert \right\} \leq (1+K_{b}e^{K_{b}})\max_{i\in \{1,\ldots
,n\}}\{\left\Vert W_{i}^{n}\right\Vert \}\text{.}
\end{equation*}%
Since $\max_{i\in \{1,\ldots ,n\}}\{\left\Vert W_{i}^{n}\right\Vert
\}\rightarrow 0$ in probability 
\begin{equation*}
\max_{i\in \{1,\ldots ,n\}}\left\{ \left\Vert \bar{Y}_{i}^{n}-\check{Y}%
_{i}^{n}\right\Vert \right\} \rightarrow 0\text{ and hence }\sup_{t\in
\lbrack 0,1]}\left\Vert \bar{Y}^{n}(t)-\check{Y}^{n}(t)\right\Vert
\rightarrow 0
\end{equation*}%
in probability.
\end{proof}

\medskip To complete the proof of the convergence we need to show $\check{Y}%
^{n}-\hat{Y}^{n}\rightarrow 0$. \ Recall that these two processes have the
same driving terms but different drifts, in that $\hat{Y}^{n}$ satisfies the
ODE 
\begin{equation*}
\hat{Y}^{n}(t)=\int_{0}^{t}Db(X^{0}(s))\hat{Y}^{n}(s)ds+\int_{0}^{t}\hat{w}%
^{n}(s)ds
\end{equation*}%
while $\check{Y}^{n}$ is the linear interpolation of the discrete time
process defined by $\check{Y}_{0}^{n}=0$ and%
\begin{equation}
\check{Y}_{i+1}^{n}=\check{Y}_{i}^{n}+\frac{a(n)}{\sqrt{n}}\left( b\left(
X_{i}^{n,0}+\frac{1}{a(n)\sqrt{n}}\check{Y}_{i}^{n}\right) -b\left(
X_{i}^{n,0}\right) \right) +\frac{1}{n}\hat{w}^{n}\left( \frac{i}{n}\right) .
\notag
\end{equation}%
However, essentially the same arguments as those used in Lemma \ref%
{lem:tight_Y_W} to show tightness of $\{\hat{Y}^{n}\}$ can be used to prove
tightness of $\{\check{Y}^{n}\}$, and then it easily follows as in Lemma \ref%
{lem:limit_Y} that any limit will satisfy the same ODE (\ref{defYhatLim}) as
the limit of $\{\hat{Y}^{n}\}$, and therefore $\check{Y}^{n}-\hat{Y}%
^{n}\rightarrow 0$ follows.

Combining $\bar{Y}^{n}-\check{Y}^{n}\rightarrow 0$, $\check{Y}^{n}-\hat{Y}%
^{n}\rightarrow 0$, and $\hat{Y}^{n}\rightarrow \hat{Y}$ demonstrates that
along the subsequence where $\hat{\eta}^{n}\rightarrow \hat{\eta}$ weakly $%
\bar{Y}^{n}\rightarrow \hat{Y}$ in distribution, which implies that along
this subsequence $(\hat{\eta}^{n},\bar{Y}^{n})\rightarrow (\hat{\eta},\hat{Y}%
)$ weakly. \ We have already shown that with probability 1 $\hat{\eta}%
_{2}(dt)$ is Lebesgue measure and 
\begin{equation*}
\hat{Y}(t)=\int_{0}^{t}Db(X^{0}(s))\hat{Y}(s)ds+\int_{0}^{t}\int_{\mathbb{R}%
^{d}}y\hat{\eta}_{\left. 1\right\vert 2}(dy\left\vert t\right. )ds,
\end{equation*}%
so the proof of convergence (i.e., the first part of Theorem \ref%
{th:Tightness}) is complete.

To finish Theorem \ref{th:Tightness} we must lastly show the bound (\ref%
{eqn:lower_bound_costs}). Note that the weak convergence of $\bar{Y}^{n}$
implies 
\begin{equation}
\sup_{t\in \lbrack 0,1]}\left\Vert \bar{X}^{n}(\left\lfloor nt\right\rfloor
/n)-X^{0}(t)\right\Vert \rightarrow 0\text{ in probability.}
\label{convLargeProc}
\end{equation}%
Now define random measures on $\mathbb{R}^{d}\otimes \mathbb{R}^{d}\otimes %
\left[ 0,1\right] $ by 
\begin{equation*}
\gamma ^{n}\left( dx\otimes dy\otimes dt\right) =\delta _{\bar{X}^{n}\left(
\left\lfloor nt\right\rfloor /n\right) }\left( dx\right) \hat{\eta}%
^{n}\left( dy\otimes dt\right) \text{.}
\end{equation*}%
Note that the tightness of $\left\{ \gamma ^{n}\right\} $ follows easily
from (\ref{convLargeProc}) and from the tightness of $\left\{ \hat{\eta}%
^{n}\right\} $. Thus given any subsequence we can choose a further
subsequence (again we will retain $n$ as the index for simplicity) along
which $\left\{ \gamma ^{n}\right\} $ converges weakly to some limit $\gamma $
on $\mathcal{P}\left( \mathbb{R}^{d}\otimes \mathbb{R}^{d}\otimes \left[ 0,1%
\right] \right) $ with 
\begin{equation*}
\gamma _{2,3}\left( dy\otimes dt\right) =\hat{\eta}\left( dy\otimes
dt\right) \text{,}
\end{equation*}%
where $\gamma _{2,3}$ is the second and third marginal of $\gamma $. \ If we
establish (\ref{eqn:lower_bound_costs}) for this subsequence it follows for
the original one using a standard argument by contradiction. For $\sigma >0$
let 
\begin{equation*}
G_{\sigma }^{X^{0}}=\left\{ \left( x,y,t\right) :\left\Vert x-X^{0}\left(
t\right) \right\Vert \leq \sigma \right\}
\end{equation*}%
be closed sets centered around $X^{0}\left( t\right) $ in the $x$ variable,
and note that by (\ref{convLargeProc}) and weak convergence, for all $\sigma
>0$ 
\begin{equation*}
1=\limsup_{n\rightarrow \infty }E\left[ \gamma ^{n}\left( G_{\sigma
}^{X^{0}}\right) \right] \leq E\left[ \gamma \left( G_{\sigma
}^{X^{0}}\right) \right] \text{.}
\end{equation*}%
Thus 
\begin{equation*}
E\left[ \gamma \left( \cap _{n\in \mathbb{N}}G_{1/n}^{X^{0}}\right) \right]
=1,
\end{equation*}%
so with probability 1 $\gamma $ puts all its mass on $\left\{ \left(
x,y,t\right) :x=X^{0}\left( t\right) \right\} $. \ Therefore with
probability 1, for a.e. $\left( y,t\right) $ under $\gamma _{2,3}\left(
dy\otimes dt\right) $, 
\begin{equation*}
\gamma _{1\left\vert 2,3\right. }\left( \left. dx\right\vert y,t\right)
=\delta _{X^{0}\left( t\right) }\left( dx\right) \text{.}
\end{equation*}%
Combined with the fact that the second marginal of $\hat{\eta}\left(
dy\otimes dt\right) $\ is Lebesgue measure, this gives 
\begin{equation}
\gamma \left( dx\otimes dy\otimes dt\right) =\delta _{X^{0}\left( t\right)
}\left( dx\right) \hat{\eta}\left( \left. dy\right\vert t\right) dt\text{.}
\label{eq:gammaDecomp}
\end{equation}

Let 
\begin{equation*}
\bar{L}_{K}\left( x,\beta \right) =\sup_{\alpha \in \mathbb{R}^{d}}\left\{
\left\langle \alpha ,\beta \right\rangle -\frac{1}{2}\left\Vert \alpha
\right\Vert _{A(x)}^{2}-\frac{1}{2K}\left\Vert \alpha \right\Vert
^{2}\right\} .
\end{equation*}%
Then uniformly in $x$ and compact subsets of $\beta $ 
\begin{equation*}
\liminf_{n\rightarrow \infty }a(n)^{2}nL_{c}\left( x,\frac{1}{a(n)\sqrt{n}}%
\beta \right) \geq \bar{L}_{K}\left( x,\beta \right) ,
\end{equation*}%
and as $K\rightarrow \infty $ 
\begin{equation*}
\bar{L}_{K}\left( x,\beta \right) \uparrow \frac{1}{2}\left\Vert \beta
\right\Vert _{A^{-1}(x)}^{2}
\end{equation*}%
for all $\left( x,\beta \right) \in \mathbb{R}^{2d}$. \ \ Combining this
with Lemma \ref{Lemma:RelEntLegIneq}\ and using Fatou's lemma for weak
convergence, 
\begin{align*}
& \liminf_{n\rightarrow \infty }a(n)^{2}nE\left[ \dsum\limits_{i=0}^{n-1}%
\frac{1}{n}R(\left. \eta _{i}^{n}\right\Vert \mu _{\bar{X}_{i}^{n}})\right]
\\
\quad & \quad \geq \liminf_{n\rightarrow \infty }E\left[ \int_{\mathbb{R}%
^{d}\otimes \mathbb{R}^{d}\otimes \left[ 0,1\right] }a(n)^{2}nL_{c}\left( x,%
\frac{1}{a(n)\sqrt{n}}y\right) \gamma ^{n}\left( dx\otimes dy\otimes
dt\right) \right] \\
& \quad \geq E\left[ \int_{\mathbb{R}^{d}\otimes \mathbb{R}^{d}\otimes \left[
0,1\right] }\bar{L}_{K}\left( x,y\right) \gamma \left( dx\otimes dy\otimes
dt\right) \right]
\end{align*}%
for all $K$. \ Then using the monotone convergence theorem, the
decomposition (\ref{eq:gammaDecomp}), and Jensen's inequality in that order
shows that 
\begin{align*}
& \liminf_{n\rightarrow \infty }a(n)^{2}nE\left[ \dsum\limits_{i=0}^{n-1}%
\frac{1}{n}R(\left. \eta _{i}^{n}\right\Vert \mu _{\bar{X}_{i}^{n}})\right]
\\
& \quad \geq \lim_{K\rightarrow \infty }E\left[ \int_{\mathbb{R}^{d}\otimes 
\mathbb{R}^{d}\otimes \left[ 0,1\right] }\bar{L}_{K}\left( x,y\right) \gamma
\left( dx\otimes dy\otimes dt\right) \right] \\
& \quad =E\left[ \int_{\mathbb{R}^{d}\otimes \mathbb{R}^{d}\otimes \left[ 0,1%
\right] }\frac{1}{2}\left\Vert y\right\Vert _{A^{-1}(x)}^{2}\gamma \left(
dx\otimes dy\otimes dt\right) \right] \\
& \quad =E\left[ \int_{0}^{1}\int_{\mathbb{R}^{d}}\frac{1}{2}\left\Vert
y\right\Vert _{A^{-1}(X^{0}\left( t\right) )}^{2}\hat{\eta}\left( \left.
dy\right\vert t\right) dt\right] \\
& \quad \geq E\left[ \frac{1}{2}\int_{0}^{1}\left\Vert \hat{w}(t)\right\Vert
_{A^{-1}(X^{0}(t))}^{2}dt\right] ,
\end{align*}%
which is (\ref{eqn:lower_bound_costs}).

\section{Laplace Upper Bound}

The goal of this section is to prove (\ref{eq:Laplace Lower}), which due to
the minus sign corresponds to the Laplace upper bound. \ Suppose for each $n$
that $\eta ^{n}$ comes within $\varepsilon $ of achieving the infimum in (%
\ref{eq:rep}), so that 
\begin{align}
& \liminf_{n\rightarrow \infty }-a(n)^{2}\log E\left[ e^{-\frac{1}{a(n)^{2}}%
F(Y^{n})}\right] +\varepsilon  \notag \\
& \quad \geq \liminf_{n\rightarrow \infty }E\left[ \dsum%
\limits_{i=0}^{n-1}a(n)^{2}R(\left. \eta _{i}^{n}\right\Vert \mu _{\bar{X}%
_{i}^{n}})+F(\bar{Y}^{n})\right] .  \label{ineq:epNearMin}
\end{align}%
\ Since $\sup_{x\in \mathbb{R}^{d}}\left\vert F(x)\right\vert \leq K_{F}$
for some $K_{F}<\infty $, we also have%
\begin{equation*}
\sup_{n}a(n)^{2}nE\left[ \dsum\limits_{i=0}^{n-1}\frac{1}{n}R(\left. \eta
_{i}^{n}\right\Vert \mu _{\bar{X}_{i}^{n}})\right] \leq 2K_{F}+\varepsilon .
\end{equation*}%
Consequently we can choose a subsequence of $\{\eta ^{n}\}$ (we retain $n$
as the index for convenience) along which the conclusions of Theorem \ref%
{th:Tightness} hold. \ Combining this with (\ref{ineq:epNearMin}) gives%
\begin{align*}
& \liminf_{n\rightarrow \infty }-a(n)^{2}\log E\left[ e^{-\frac{1}{a(n)^{2}}%
F(Y^{n})}\right] +\varepsilon \\
& \quad \geq \liminf_{n\rightarrow \infty }E\left[ \dsum%
\limits_{i=0}^{n-1}a(n)^{2}R(\left. \eta _{i}^{n}\right\Vert \mu _{\bar{X}%
_{i}^{n}})+F(\bar{Y}^{n})\right] \\
& \quad \geq E\left[ \dint_{0}^{1}\frac{1}{2}\left\Vert \hat{w}%
(s)\right\Vert _{A^{-1}(X^{0}(s))}^{2}ds+F(\hat{Y})\right] \text{.}
\end{align*}%
Recalling%
\begin{equation*}
\hat{Y}(t)=\int_{0}^{t}Db(X^{0}(s))\hat{Y}(s)ds+\int_{0}^{t}\hat{w}(s)ds,
\end{equation*}%
it follows that 
\begin{align*}
& E\left[ \dint_{0}^{1}\frac{1}{2}\left\Vert \hat{w}(s)\right\Vert
_{A^{-1}(X^{0}(s))}^{2}ds+F(\hat{Y})\right] \\
\quad & \geq \inf_{u\in L^{2}([0,1]:\mathbb{R}^{d})}\left\{ \dint_{0}^{1}%
\frac{1}{2}\left\Vert u(s)\right\Vert _{A^{-1}(X^{0}(s))}^{2}ds+F(\phi
^{u})\right\} \\
\quad & =\inf_{u\in L^{2}([0,1]:\mathbb{R}^{d})}\left\{ \dint_{0}^{1}\frac{1%
}{2}\left\Vert u(s)\right\Vert ^{2}ds+F\left( \phi ^{A^{1/2}(X^{0})u}\right)
\right\} ,
\end{align*}%
with $\phi ^{u}$ defined as in (\ref{phiDefined}). Since $\varepsilon >0$ is
arbitrary, we have the lower bound (\ref{eq:Laplace Lower}).

\section{Laplace Lower Bound}

The goal of this section is to prove (\ref{eq: Laplace Upper}). \ Note that
for $u,v\in L^{2}([0,1]:\mathbb{R}^{d})$%
\begin{align*}
& \phi ^{A^{1/2}(X^{0})u}(t)-\phi ^{A^{1/2}(X^{0})v}(t) \\
& \quad =\int_{0}^{t}Db(X^{0}(s))\left( \phi ^{A^{1/2}(X^{0})u}(s)-\phi
^{A^{1/2}(X^{0})v}(s)\right) ds \\
& \qquad +\int_{0}^{t}A^{1/2}(X^{0}(s))(u(s)-v(s))ds.
\end{align*}%
Thus by Gronwall's inequality 
\begin{align}
& \sup_{t\in \lbrack 0,1]}\left\Vert \phi ^{A^{1/2}(X^{0})u}(t)-\phi
^{A^{1/2}(X^{0})v}(t)\right\Vert  \notag \\
& \quad \leq (1+K_{b}e^{K_{b}})K_{A}^{1/2}\int_{0}^{1}\left\Vert
u(s)-v(s)\right\Vert ds  \notag \\
& \quad \leq (1+K_{b}e^{K_{b}})K_{A}^{1/2}\left( \int_{0}^{1}\left\Vert
u(s)-v(s)\right\Vert ^{2}ds\right) ^{\frac{1}{2}}\text{.}
\label{eq:L2ContPhi}
\end{align}%
Since $C([0,1]:\mathbb{R}^{d})$\ is dense in $L^{2}([0,1]:\mathbb{R}^{d})$,
the proof of the Laplace lower bound is reduced to showing that for an
arbitrary $u\in C([0,1]:\mathbb{R}^{d})$ 
\begin{equation}
\limsup_{n\rightarrow \infty }-a(n)^{2}\log E\left[ e^{-\frac{1}{a(n)^{2}}%
F(Y^{n})}\right] \leq \frac{1}{2}\int_{0}^{1}\left\Vert u(s)\right\Vert
^{2}ds+F\left( \phi ^{A^{1/2}(X^{0})u}\right) \text{.}  \label{eq:last}
\end{equation}

The main difficulty is to deal with the possible degeneracy of the noise.
Recall the orthogonal decomposition of $A^{-1}(x)$ (\ref{eq:DiffInvDecomp}).
\ Define%
\begin{equation*}
A_{K}^{-1}(x)=Q(x)\Lambda _{K}^{-1}(x)Q^{T}(x)
\end{equation*}%
where $\Lambda _{K}^{-1}(x)$ is the diagonal matrix such that $\Lambda
_{ii,K}^{-1}(x)=\Lambda _{ii}^{-1}(x)$ when $\Lambda _{ii}^{-1}(x)\leq K^{2}$
and $\Lambda _{ii,K}^{-1}(x)=K^{2}$ when $\Lambda _{ii}^{-1}(x)>K^{2}$. \
Note that by \cite[Theorem 6.2.37]{horjoh} $A^{1/2}(x)$, $A_{K}^{-1}(x)$ and 
$A_{K}^{1/2}(x)$ are continuous functions of $A(x)$, and consequently they
are also continuous functions of $x\in \mathbb{R}^{d}$. \ In addition define%
\begin{equation*}
u_{K}(s)=\left\{ 
\begin{array}{cc}
u(s) & \text{for }\left\Vert u(s)\right\Vert \leq K \\ 
\frac{Ku(s)}{\left\Vert u(s)\right\Vert } & \text{for }\left\Vert
u(s)\right\Vert >K%
\end{array}%
\right. \text{ .}
\end{equation*}%
Let $\phi ^{u,K}(t)=\phi ^{A(X^{0})A_{K}^{-1/2}(X^{0})u_{K}}(t)$, and note
that $\phi ^{u,K}$ solves \ 
\begin{align}
\phi ^{u,K}(t)& =\int_{0}^{t}Db(X^{0}(s))\phi ^{u,K}(s)ds  \notag \\
& \quad +\int_{0}^{t}A(X^{0}(s))A_{K}^{-1/2}(X^{0}(s))u_{K}(s)ds\text{.}
\label{eq:PhiuK}
\end{align}

To simplify notation we define $s_{i}^{n}\doteq i/n$ and $%
s^{n}(t)=\left\lfloor nt\right\rfloor /n$, where $\left\lfloor
a\right\rfloor $ is the integer part of $a$. Note that $s^{n}(t)-t%
\rightarrow 0$ uniformly for $t\in \lbrack 0,1]$ as $n\rightarrow \infty $.
For $n$ sufficiently large 
\begin{equation*}
\max_{0\leq i\leq n-1}\left\{ \frac{1}{a(n)\sqrt{n}}\left\Vert
A_{K}^{-1/2}\left( X^{0}\left( s_{i}^{n}\right) \right) u_{K}\left(
s_{i}^{n}\right) \right\Vert \right\} \leq \frac{1}{a(n)\sqrt{n}}K^{2}\leq
\lambda _{DA}
\end{equation*}%
and we can define the sequence $\{(\bar{X}^{n,u,K},\bar{Y}^{n,u,K},\eta
^{n,u,K},\hat{\eta}^{n,u,K})\}$ as in Construction \ref{tightnessNotation}
with%
\begin{align*}
& \eta _{i}^{n,u,K}(dy) \\
& \quad =\exp \left\{ \left\langle y,\frac{1}{a(n)\sqrt{n}}%
A_{K}^{-1/2}\left( X^{0}\left( s_{i}^{n}\right) \right) u_{K}\left(
s_{i}^{n}\right) \right\rangle \right. \\
& \qquad \qquad \left. -H_{c}\left( \bar{X}_{i}^{n,u,K},\frac{1}{a(n)\sqrt{n}%
}A_{K}^{-1/2}\left( X^{0}\left( s_{i}^{n}\right) \right) u_{K}\left(
s_{i}^{n}\right) \right) \right\} \mu _{\bar{X}_{i}^{n,u,K}}(dy)\text{.}
\end{align*}%
\ Using (\ref{eq:H3rdDerivBnd}) and the fact that 
\begin{equation*}
\dint_{\mathbb{R}^{d}}y\exp \{\left\langle y,\alpha \right\rangle
-H_{c}(x,\alpha )\}\mu _{x}(dy)=D_{\alpha }H_{c}(x,\alpha )
\end{equation*}%
we have for $\left\Vert \alpha \right\Vert \leq \lambda _{DA}$ 
\begin{equation}
\left\Vert \dint_{\mathbb{R}^{d}}y\exp \{\left\langle y,\alpha \right\rangle
-H_{c}(x,\alpha )\}\mu _{x}(dy)-A(x)\alpha \right\Vert \leq K_{DA}\left\Vert
\alpha \right\Vert ^{2}\text{.}  \label{Derivative Bound}
\end{equation}%
The next result identifies the limit in probability of the controlled
processes and an asymptotic bound for the relative entropies.

\begin{theorem}
\label{th:ConvToPhiK}Let $u\in C(\left[ 0,1\right] :\mathbb{R}^{d})$ and $%
K<\infty $ be given, construct $\{(\bar{X}^{n,u,K},\bar{Y}^{n,u,K},\eta
^{n,u,K},\hat{\eta}^{n,u,K})\}$ as in this section and define $\phi ^{u,K}$
by (\ref{eq:PhiuK}). Then 
\begin{equation}
\bar{Y}^{n,u,K}\rightarrow \phi ^{u,K}  \label{eq:convofY}
\end{equation}%
in $C([0,1]:\mathbb{R}^{d})$ in probability, and 
\begin{align}
& \limsup_{n\rightarrow \infty }a^{2}(n)nE\left[ \frac{1}{n}%
\dsum\limits_{i=0}^{n-1}R\left( \left. \eta _{i}^{n,u,K}\right\Vert \mu _{%
\bar{X}_{i}^{n,u,K}}\right) \right]  \notag \\
& \quad \leq \frac{1}{2}\int_{0}^{1}\left\Vert
A_{K}^{-1/2}(X^{0}(s))u_{K}(s)\right\Vert _{A(X^{0}(s))}^{2}ds.
\label{eq:REbound2}
\end{align}
\end{theorem}

\begin{proof}
From (\ref{eq:H-bound}) and (\ref{Derivative Bound}) we have for all $n$
sufficiently large that $\frac{1}{a(n)\sqrt{n}}K^{2}\leq \lambda _{DA}$ 
\begin{align*}
& R\left( \left. \eta _{i}^{n,u,K}\right\Vert \mu _{\bar{X}%
_{i}^{n,u,K}}\right) \\
& \quad =\dint_{\mathbb{R}^{d}}\left\langle y,\frac{1}{a(n)\sqrt{n}}%
A_{K}^{-1/2}\left( X^{0}\left( s_{i}^{n}\right) \right) u_{K}\left(
s_{i}^{n}\right) \right\rangle \eta _{i}^{n,u,K}(dy) \\
& \qquad -H_{c}\left( \bar{X}_{i}^{n,u,K},\frac{1}{a(n)\sqrt{n}}%
A_{K}^{-1/2}\left( X^{0}\left( s_{i}^{n}\right) \right) u_{K}\left(
s_{i}^{n}\right) \right) \\
& \quad \leq \left\langle \frac{1}{a(n)\sqrt{n}}A\left( \bar{X}%
_{i}^{n,u,K}\right) A_{K}^{-1/2}\left( X^{0}\left( s_{i}^{n}\right) \right)
u_{K}\left( s_{i}^{n}\right) ,\right. \\
& \qquad \qquad \qquad \left. \frac{1}{a(n)\sqrt{n}}A_{K}^{-1/2}\left(
X^{0}\left( s_{i}^{n}\right) \right) u_{K}\left( s_{i}^{n}\right)
\right\rangle \\
& \qquad -\frac{1}{2}\left\langle \frac{1}{a(n)\sqrt{n}}A\left( \bar{X}%
_{i}^{n,u,K}\right) A_{K}^{-1/2}\left( X^{0}\left( s_{i}^{n}\right) \right)
u_{K}\left( s_{i}^{n}\right) ,\right. \\
& \qquad \qquad \qquad \left. \frac{1}{a(n)\sqrt{n}}A_{K}^{-1/2}\left(
X^{0}\left( s_{i}^{n}\right) \right) u_{K}\left( s_{i}^{n}\right)
\right\rangle +\frac{2}{a(n)^{3}n^{3/2}}K_{DA}K^{6} \\
& \quad =\frac{1}{2a(n)^{2}n}\left\Vert A_{K}^{-1/2}\left( X^{0}\left(
s_{i}^{n}\right) \right) u_{K}\left( s_{i}^{n}\right) \right\Vert _{A(\bar{X}%
_{i}^{n,u,K})}^{2}+\frac{2}{a(n)^{3}n^{3/2}}K_{DA}K^{6}.
\end{align*}%
Consequently%
\begin{align}
& \limsup_{n\rightarrow \infty }a^{2}(n)nE\left[ \frac{1}{n}%
\dsum\limits_{i=0}^{n-1}R\left( \left. \eta _{i}^{n,u,K}\right\Vert \mu _{%
\bar{X}_{i}^{n,u,K}}\right) \right]  \label{ineq:GaussUpper} \\
& \quad \leq \limsup_{n\rightarrow \infty }\frac{1}{2}E\left[ \frac{1}{n}%
\dsum\limits_{i=0}^{n-1}\left\Vert A_{K}^{-1/2}\left( X^{0}\left(
s_{i}^{n}\right) \right) u_{K}\left( s_{i}^{n}\right) \right\Vert _{A(\bar{X}%
_{i}^{n,u,K})}^{2}\right] ,  \notag
\end{align}%
where in fact 
\begin{equation*}
\limsup_{n\rightarrow \infty }\frac{1}{2}E\left[ \frac{1}{n}%
\dsum\limits_{i=0}^{n-1}\left\Vert A_{K}^{-1/2}\left( X^{0}\left(
s_{i}^{n}\right) \right) u_{K}\left( s_{i}^{n}\right) \right\Vert _{A(\bar{X}%
_{i}^{n,u,K})}^{2}\right] \leq \frac{1}{2}K^{4}K_{A}\text{.}
\end{equation*}%
Therefore (\ref{cond:relEndBnd}) is satisfied by $\{\eta ^{n,u,K}\}$, so we
can apply Theorem \ref{th:Tightness} and choose a subsequence (keeping $n$
as the index for convenience) along which $\{(\hat{\eta}^{n,u,K},\bar{Y}%
^{n,u,K})\}$ converges weakly to some limit $(\hat{\eta}^{u,K},\hat{Y}%
^{u,K}) $, where $\hat{\eta}_{2}^{u,K}$ is Lebesgue measure and%
\begin{equation*}
\hat{Y}^{u,K}(t)=\int_{0}^{t}Db(X^{0}(s))\hat{Y}^{u,K}(s)ds+\int_{0}^{t}%
\int_{\mathbb{R}^{d}}y\hat{\eta}_{\left. 1\right\vert 2}^{u,K}(dy\left\vert
s\right. )ds\text{.}
\end{equation*}%
This implies 
\begin{equation}
\sup_{t\in \lbrack 0,1]}\left\Vert \bar{X}^{n,u,K}(t)-X^{0}(t)\right\Vert
\rightarrow 0  \label{XConverge}
\end{equation}%
in probability. \ Because of this, the uniform bound on $A^{1/2}(x)$ and the
continuity of $A^{1/2}(x)$, we have (recall that $s^{n}(t)\doteq
\left\lfloor nt\right\rfloor /n$) 
\begin{equation*}
\sup_{t\in \lbrack 0,1]}\left\Vert A^{1/2}(\bar{X}%
^{n,u,K}(s^{n}(t)))-A^{1/2}(X^{0}(s^{n}(t)))\right\Vert \rightarrow 0
\end{equation*}%
in probability. \ However, the continuity of $%
A^{1/2}(X^{0})A_{K}^{-1/2}(X^{0})u_{K}$ gives%
\begin{align*}
& \sup_{t\in \lbrack 0,1]}\left\Vert
A^{1/2}(X^{0}(s^{n}(t)))A_{K}^{-1/2}(X^{0}(s^{n}(t)))u_{K}(s^{n}(t))\right.
\\
& \qquad \left. -A^{1/2}(X^{0}(t))A_{K}^{-1/2}(X^{0}(t))u_{K}(t)\right\Vert
\rightarrow 0\text{.}
\end{align*}%
Combining these limits, and using the fact that $A_{K}^{-1/2}(X^{0})u_{K}$
is uniformly bounded, shows that%
\begin{align}
& \sup_{t\in \lbrack 0,1]}\left\Vert A^{1/2}(\bar{X}%
^{n,u,K}(s^{n}(t)))A_{K}^{-1/2}(X^{0}(s^{n}(t)))u_{K}(s^{n}(t))\right.
\label{eq:truncTiltConvProb} \\
& \qquad \left. -A^{1/2}(X^{0}(t))A_{K}^{-1/2}(X^{0}(t))u_{K}(t)\right\Vert
\rightarrow 0  \notag
\end{align}%
in probability. \ This combined with the uniform bounds allows us to use
dominated convergence to get \ 
\begin{align*}
& \limsup_{n\rightarrow \infty }E\left[ \frac{1}{2}\int_{0}^{1}\left\Vert
A_{K}^{-1/2}(X^{0}(s^{n}(t)))u_{K}(s^{n}(t))\right\Vert _{A(\bar{X}%
^{n,u,K}(s^{n}(t)))}^{2}dt\right] \\
& \quad =\frac{1}{2}\int_{0}^{1}\left\Vert
A_{K}^{-1/2}(X^{0}(t))u_{K}(t)\right\Vert _{A(X^{0}(t))}^{2}dt.
\end{align*}%
Combining this with (\ref{ineq:GaussUpper}) shows (\ref{eq:REbound2}).

To prove (\ref{eq:convofY}) we will show that in fact 
\begin{equation*}
\hat{\eta}^{u,K}(dy\otimes dt)=\delta
_{A(X^{0}(t))A_{K}^{-1/2}(X^{0}(t))u_{K}(t)}(dy)dt\text{.}
\end{equation*}%
For all $\sigma >0$ let 
\begin{equation*}
G_{\sigma }=\left\{ (z,t)\in \mathbb{R}^{d}\times \lbrack 0,1]:\left\Vert
z-A(X^{0}(t))A_{K}^{-1/2}(X^{0}(t))u_{K}(t)\right\Vert \leq \sigma \right\} ,
\end{equation*}%
and note that by weak convergence $\limsup_{n\rightarrow \infty }E[\hat{\eta}%
^{n,u,K}(G_{\sigma })]\leq E[\hat{\eta}^{u,K}(G_{\sigma })]$. Note also that%
\begin{align*}
& E[\hat{\eta}^{n,u,K}(G_{\sigma })] \\
& \quad \geq P\left[ \sup_{t\in \lbrack 0,1]}\left\Vert a(n)\sqrt{n}\int_{%
\mathbb{R}^{d}}y\eta _{\left\lfloor nt\right\rfloor
}^{n,u,K}(dy)-A(X^{0}(t))A_{K}^{-1/2}(X^{0}(t))u_{K}(t)\right\Vert \leq
\sigma \right] \text{.}
\end{align*}%
However, by (\ref{Derivative Bound}) we can choose $n$ large enough to make 
\begin{align*}
& \sup_{t\in \lbrack 0,1]}\left\Vert a(n)\sqrt{n}\int_{\mathbb{R}^{d}}y\eta
_{\left\lfloor nt\right\rfloor }^{n,u,K}(dy)\right. \\
& \qquad \quad \left. -A\left( \bar{X}^{n,u,K}\left( s^{n}(t)\right) \right)
A_{K}^{-1/2}\left( X^{0}\left( s^{n}(t)\right) \right) u_{K}\left(
s^{n}(t)\right) \rule{0pt}{12pt}\rule{0pt}{14pt}\right\Vert
\end{align*}%
arbitrarily small, and the proof that 
\begin{align*}
& \sup_{t\in \lbrack 0,1]}\left\Vert A(\bar{X}%
^{n,u,K}(s^{n}(t)))A_{K}^{-1/2}(X^{0}(s^{n}(t)))u_{K}(s^{n}(t))\right. \\
& \qquad \quad \left. -A(X^{0}(t))A_{K}^{-1/2}(X^{0}(t))u_{K}(t)\right\Vert
\rightarrow 0
\end{align*}%
in probability is identical to the proof of (\ref{eq:truncTiltConvProb}). \
Therefore $\limsup_{n\rightarrow \infty }E[\hat{\eta}^{u,K,n}(G_{\sigma
})]=1 $ for all $\sigma >0$, and so $E[\hat{\eta}^{u,K}(\cap _{n\in \mathbb{N%
}}G_{1/n})]=1$. This implies that with probability 1 
\begin{equation*}
\hat{\eta}_{\left. 1\right\vert 2}^{u,K}(\left. dy\right\vert t)=\delta
_{A(X^{0}(t))A_{K}^{-1/2}(X^{0}(t))u_{K}}(dy)
\end{equation*}%
for a.e. $t$. It follows that 
\begin{equation*}
\hat{Y}^{u,K}(t)=\int_{0}^{t}Db(X^{0}(s))\hat{Y}^{u,K}(s)ds+%
\int_{0}^{t}A(X^{0}(s))A_{K}^{-1/2}(X^{0}(s))u_{K}(s)ds,
\end{equation*}%
and therefore $\bar{Y}^{n,u,K}\rightarrow \phi ^{u,K}$ weakly. This implies (%
\ref{eq:convofY}) and completes the proof.
\end{proof}

\medskip The second theorem in this section allows us to approximate $F(\phi
^{A^{1/2}(X^{0})u})$ by $F(\phi ^{u,K})$ and $\frac{1}{2}\int_{0}^{1}\left%
\Vert u(s)\right\Vert ^{2}ds$ by 
\begin{equation*}
\frac{1}{2}\int_{0}^{1}\left\Vert A_{K}^{-1/2}(X^{0}(s))u_{K}(s)\right\Vert
_{A(X^{0}(s))}^{2}ds\text{.}
\end{equation*}

\begin{theorem}
\label{th:ConvPhiKToPhi}Let $u\in C([0,1]:\mathbb{R}^{d})$ and define $\phi
^{A_{K}^{1/2}(X^{0})u}$ by (\ref{phiDefined}) and $\phi ^{u,K}$ by (\ref%
{eq:PhiuK}). \ Then 
\begin{equation*}
\phi ^{u,K}\rightarrow \phi ^{A^{1/2}(X^{0})u}
\end{equation*}%
in $C([0,1]:\mathbb{R}^{d})$ and 
\begin{equation*}
\frac{1}{2}\int_{0}^{1}\left\Vert A_{K}^{-1/2}(X^{0}(s))u_{K}(s)\right\Vert
_{A(X^{0}(s))}^{2}ds\rightarrow \frac{1}{2}\int_{0}^{1}\left\Vert
u(s)\right\Vert ^{2}ds\text{.}
\end{equation*}
\end{theorem}

\begin{proof}
Note that 
\begin{equation*}
\left\Vert A^{1/2}(X^{0}(s))A_{K}^{-1/2}(X^{0}(s))u_{K}(s)\right\Vert \leq
\left\Vert u(s)\right\Vert
\end{equation*}%
for all $s\in \lbrack 0,1]$ and 
\begin{equation}
A^{1/2}(X^{0}(s))A_{K}^{-1/2}(X^{0}(s))u_{K}(s)\rightarrow u(s)
\label{eq:le}
\end{equation}%
pointwise. Since $u\in L^{2}([0,1]:\mathbb{R}^{d})$, by dominated
convergence (\ref{eq:le}) also holds in $L^{2}([0,1]:\mathbb{R}^{d})$. This
gives 
\begin{equation*}
\frac{1}{2}\int_{0}^{1}\left\Vert A_{K}^{-1/2}(X^{0}(s))u_{K}(s)\right\Vert
_{A(X^{0}(s))}^{2}ds\rightarrow \frac{1}{2}\int_{0}^{1}\left\Vert
u(s)\right\Vert ^{2}ds\text{.}
\end{equation*}%
Combining this with (\ref{eq:L2ContPhi}) shows that 
\begin{equation*}
\phi ^{u,K}\rightarrow \phi ^{A^{1/2}(X^{0})u}
\end{equation*}%
in $C([0,1]:\mathbb{R}^{d})$.
\end{proof}

\medskip \medskip Using (\ref{eq:rep}) and the fact that any given control
is suboptimal, 
\begin{align*}
& -a(n)^{2}\log E\left[ e^{-\frac{1}{a(n)^{2}}F(Y^{n})}\right] \\
& \quad \leq E\left[ \dsum\limits_{i=0}^{n-1}a(n)^{2}R\left( \left. \eta
_{i}^{n,u,K}\right\Vert \mu _{\bar{X}_{i}^{n,u,K}}\right) +F(\bar{Y}^{n,u,K})%
\right] \text{.}
\end{align*}%
Using Theorem \ref{th:ConvToPhiK}, this implies 
\begin{align*}
& \limsup_{n\rightarrow \infty }-a(n)^{2}\log E\left[ e^{-\frac{1}{a(n)^{2}}%
F(Y^{n})}\right] \\
& \quad \leq \frac{1}{2}\int_{0}^{1}\left\Vert
A_{K}^{-1/2}(X^{0}(s))u_{K}(s)\right\Vert _{A(X^{0}(s))}^{2}ds+F(\phi ^{u,K})%
\text{.}
\end{align*}%
Sending $K\rightarrow \infty $ and using Theorem \ref{th:ConvPhiKToPhi}
gives (\ref{eq:last}), and hence completes the proof of the lower bound (\ref%
{eq: Laplace Upper}).


\end{document}